\documentclass[10pt]{article}
\usepackage{amssymb,stmaryrd,amsmath,amsfonts,amsthm,enumerate,color}
\usepackage[toc,page,title,titletoc,header]{appendix}
\usepackage{graphicx}
\usepackage{upgreek}
\usepackage{algorithm}
\usepackage[noend]{algpseudocode}
\usepackage{savesym}
\usepackage{comment}
\usepackage{mathtools}
\usepackage{mathrsfs}
\usepackage{subfigure}
\usepackage[latin1]{inputenc}
\usepackage[T1]{fontenc}
\usepackage{indentfirst}
\usepackage{eurosym}
\usepackage[shortlabels]{enumitem}
\usepackage{multicol}
\usepackage{tikz-cd} 
\usepackage{bbm}
\usepackage{booktabs}
\usepackage{hyperref}
\usepackage[T1]{fontenc}
\usepackage[dvipsnames]{xcolor}
\usepackage{empheq}
\usepackage[hypcap=false]{caption}

\numberwithin{equation}{section}

\newtheorem{Theorem}{Theorem}[section]
\newtheorem{Lemma}[Theorem]{Lemma}

\newtheorem{Example}[Theorem]{Example}

\newtheorem{Definition}[Theorem]{Definition}

\newtheorem{Remark}[Theorem]{Remark}
\numberwithin{equation}{section}
\def\O{\mathcal{O}}
\def\hbn{\hat{\n}}

\def\G{\mathcal G}
\def\Sr{\mathcal S}

\def\T{\mathcal{T}}
\def\embed{\hookrightarrow}
\def\la{\langle}
\def\ra{\rangle}
\def\D{{\mathcal D}}
\def\B{{\mathcal{B}}}
\def\L{\mathscr{L}}

\def\K{{\mathcal K}}

\def\F{{\mathcal F}}

\def\M{{\mathcal{M}}}

\def\Lom3{L^2(\Omega)^3}

\def\Om{\Omega}

    \def\R{{\mathbb R}}

\def\p{\partial}
 
\def\n{{\bf n}}   \def\A{{\mathcal{ A}}}

\newcommand{\lc}
{\mathrel{\raise2pt\hbox{${\mathop<\limits_{\raise1pt\hbox
{\mbox{$\sim$}}}}$}}}

\newcommand{\gc}
{\mathrel{\raise2pt\hbox{${\mathop>\limits_{\raise1pt\hbox{\mbox{$\sim$}}}}$}}}

\newcommand{\ec}
{\mathrel{\raise2pt\hbox{${\mathop=\limits_{\raise1pt\hbox{\mbox{$\sim$}}}}$}}}

\def\bb{\begin{equation}} \def\ee{\end{equation}}

\def\beqn{\begin{eqnarray}}  \def\eqn{\end{eqnarray}}
\def\beq{\begin{equation}} \def\eeq{\end{equation}}

\def\beqnx{\begin{eqnarray*}} \def\eqnx{\end{eqnarray*}}

\def\bn{\begin{enumerate}} \def\en{\end{enumerate}}

\def\bd{\begin{description}} \def\ed{\end{description}}

\makeatletter

\newenvironment{figurehere}
  {\def\@captype{figure}}
  {}
\makeatother

\title{\bf {Derivative-Free Recovery of a Nonlinearity in a Free-Boundary DCIS Model}}

\author{ 
De-Han Chen\footnote{School of Mathematics and Statistics, and Hubei Key Laboratory of Mathematical Sciences, Central China Normal University, P.O. Box 71010, Wuhan 430079, China ({\tt dhchen@ccnu.edu.cn}). The work of DC was supported by NNSF of China (No. 11701205), and DFG research grant YO 159/5-1, project number: 513566305.}
\and Hongyu Liu\footnote{Corresponding author, Department of Mathematics, City University of Hong Kong, Kowloon, Hong Kong, China ({\tt hongyliu@cityu.edu.hk}).}
\and Keji Liu\footnote{Corresponding author, School of Mathematics, and Dishui Lake Advanced Finance Institute, Shanghai University of Finance and Economics, Shanghai 200433, China ({\tt liu.keji@sufe.edu.cn}).}
}

\begin{document}

\date{}
\maketitle

\begin{abstract}

We investigate an inverse coefficient problem for a multidimensional free-boundary model of ductal carcinoma in situ (DCIS), in which the tumor interface is governed by the nonlinear coupling of nutrient concentration, tissue pressure and curvature, and the unknown nutrient consumption function is recovered from a temporal trace of the nutrient concentration obtained by needle aspiration biopsy. For the forward problem, we establish uniform local well-posedness over an admissible class of consumption functions. The inverse problem is recast as a fixed-point problem: approximating the admissible set by finite-dimensional spaces yields discrete iteration operators, for which we prove the existence of fixed points, and the strong convergence of a subsequence of discrete fixed points to a fixed point of the continuous operator, which solves the inverse problem under a consistency condition. To approximate these fixed points, we develop a homotopy-continuation method combining a linearly convergent Picard iteration with a cubical Sperner search, without differentiating an objective functional or computing an adjoint state. Several numerical experiments on radially symmetric and  non-symmetric DCIS models corroborate the theoretical findings.

\medskip
		
		

\end{abstract}

\section{Introduction}
Ductal carcinoma in situ (DCIS) is commonly recognized as the earliest detectable phase of breast malignancy, wherein neoplastic epithelial cells undergo uncontrolled multiplication while remaining strictly confined within the ductal lumen by an uninterrupted basement membrane. A visual representation of this architecture is shown in Figure~\ref{fig0}. Due to its non-invasive nature during this stage, DCIS typically does not present with overt clinical symptoms; instead, it is most frequently identified by mammographic imaging, whether through routine population-based screening programs or following the palpation of a suspicious mass.
Prompt therapeutic intervention is essential in the treatment of DCIS, as prolonged inaction may allow the eventual rupture of the ductal wall and subsequent infiltration of malignant cells into the adjacent stromal tissue--a transition that marks the onset of invasive breast carcinoma. In this context, mathematical modeling serves as a powerful prognostic tool, capable of simulating the temporal and spatial evolution of DCIS and offering quantitative forecasts regarding the likelihood of progression from atypical ductal hyperplasia to fully invasive cancer. This predictive capacity enables clinicians to tailor intervention strategies with greater timeliness and specificity.
Historically, Greenspan's pioneering contributions in the 1970s laid the groundwork for the quantitative study of tumor growth \cite{greenspan1972models}. Based on this foundation, Byrne and Chaplain formulated the first dedicated mathematical description of the dynamics of DCIS \cite{byrne1995growth}. Their framework has since undergone extensive refinement and validation over more than thirty years, as documented in a substantial body of literature \cite{ChenFried2003, cui2007well, cui2008, cui2009lie, cui2013, escher2004classical}. Currently, a range of analogous free-boundary problems arising in tumor biology has been proposed and rigorously examined in \cite{ChenFried2003, cui2008, cui2013}.

One typical feature of the DCIS mathematical framework lies in its free boundary condition. This implies that the spatial domain occupied by neoplastic cells is not fixed but evolves temporally, and its evolving interface cannot be predetermined. In addition, this boundary is not known a priori and exhibits a complex dependence on both the local mean curvature and the prevailing nutrient distribution.
Suppose that the tumor region is denoted by $\Omega(t)\subset \R^n$ ($n\geq 2$)  and its boundary 
$$
\p\Omega(t)=\Sigma\cup \Gamma(t) \quad \forall t\geq 0
$$ 
consists of two disjoint parts: a fixed outer boundary 
$\Sigma$  and a moving internal boundary $\Gamma(t)$, see Figure \ref{fig0}(c).  

\begin{figure}[ht]
\vskip -0.3truecm
\centering
% 第一行：三个图均匀分布
\begin{minipage}[t]{0.16\textwidth}
\centering
\includegraphics[clip,width=\textwidth]{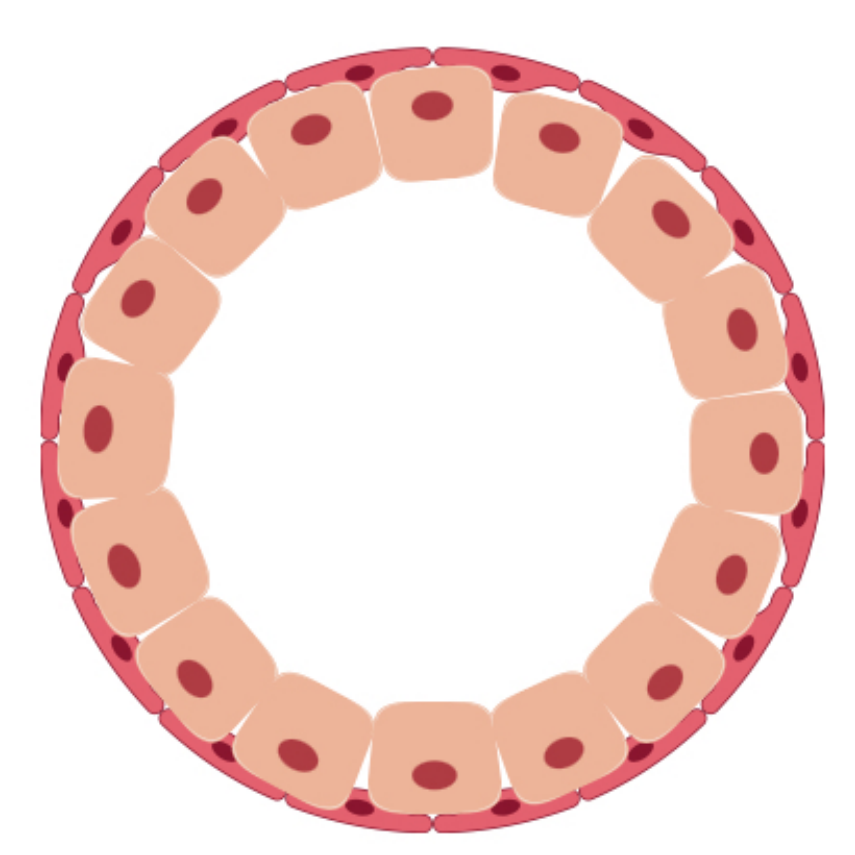}
\centering \vspace{0.3em}
(a)
\end{minipage}%
\hspace{1.4cm}
\begin{minipage}[t]{0.16\textwidth}
\centering
\includegraphics[clip,width=\textwidth]{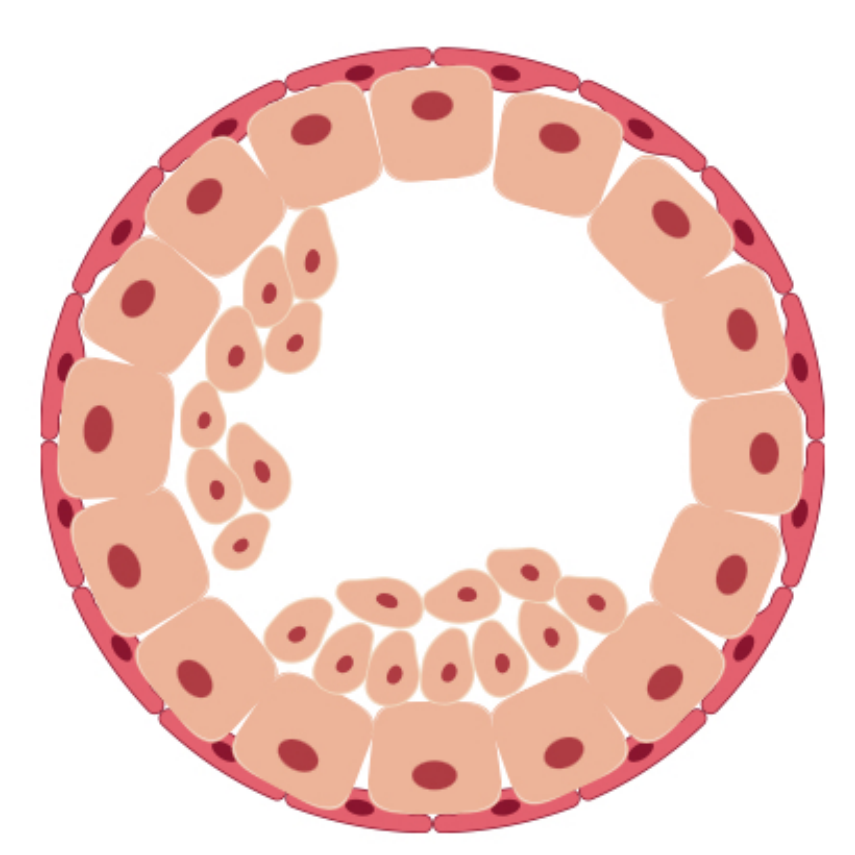}
\vspace{0.3em}
\centering
(b)
\end{minipage}%
\hspace{0.8cm}
\begin{minipage}[t]{0.30\textwidth}
\centering
\begin{tikzpicture}[scale=0.55]
 \useasboundingbox (-2.6,-2.3) rectangle (2.6,2.3); 
    \draw [red,ultra thick] (0,0) circle (2.2cm);
    \fill[orange!40!white] (0,0) circle (2.2cm);
    \fill [white] (0,1.5)
      to [out=0,in=90] (1.5,0)
      to [out=-90,in=45] (1.06,-1.06)
      to [out=-135,in=-45] (0,-0.85)
      to [out=180,in=45] (-0.85,-1)
      to [out=180,in=-45] (-1.06,-0.3)
      to [out=135,in=-90] (-0.75,0.75)
      to [out=90,in=-180] (0,1.5);
    \draw [orange,very thick] (0,1.5)
      to [out=0,in=90] (1.5,0)
      to [out=-90,in=45] (1.06,-1.06)
      to [out=-135,in=-45] (0,-0.85)
      to [out=180,in=45] (-0.85,-1)
      to [out=180,in=-45] (-1.06,-0.3)
      to [out=135,in=-90] (-0.75,0.75)
      to [out=90,in=-180] (0,1.5);
    \draw (-0.3,0) node {$\Gamma(t)$};
    \draw (-0.5,-1.65) node {$\Omega(t)$};
    \draw (2.6,0) node {$\Sigma$};
\end{tikzpicture}
\vspace{0.3em}

(c)
\end{minipage}
\caption{\label{fig0} \emph {The demonstrations of (a) normal duct, (b) DCIS and (c) mathematical model for DCIS.}}
\vskip -0.4truecm
\end{figure}
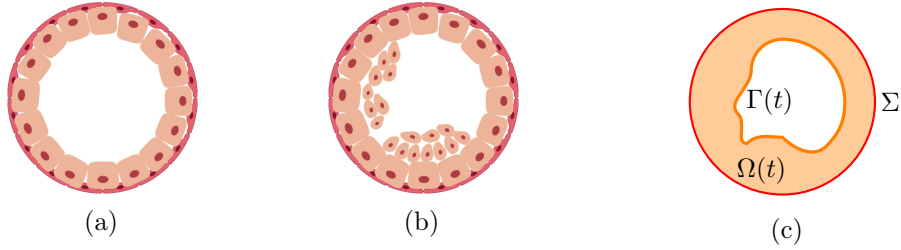
 The DCIS model consists of the following parabolic-elliptic system for the
concentration of nutrients $\sigma=\sigma(x,t)$ and the pressure of tumor
tissue $\varpi=\varpi(x,t)$:
\begin{empheq}[left=\empheqlbrace]{alignat=3}
\partial_t\sigma-\Delta\sigma &= -f(\sigma), \quad &
-\Delta\varpi &= \mu(\sigma-\sigma^\star)
&&\quad\text{in }\Omega(t),\notag\\[1mm]
V_{\mathbf{n}} &= -\partial_{\mathbf{n}}\varpi, \quad &
\varpi &= \gamma\kappa_\Gamma,\quad \sigma = \sigma_b
&&\quad\text{on }\Gamma(t),\label{eq:DCIS:full}\\[1mm]
\partial_\n\varpi &= 0, \quad &
\partial_\n\sigma &= 0
&&\quad \text{on }\Sigma,\notag\\[1mm]
\Gamma(0) &= \Gamma_e, \quad &
\sigma(\cdot,0) &= \sigma_0(\cdot)
&&\quad \text{in }\Omega_0.\notag
\end{empheq}
where $\mu,\sigma^\star>0$ are known constants, the constant
$\sigma_b$ represents the prescribed nutrient concentration on the moving boundary $\Gamma(t)$, the constant $\gamma>0$ denotes the surface tension, $\Omega_0$ denotes
the domain initially occupied by the tumor, $\Gamma_e$ is the initial shape
of the free boundary, and $\sigma_0$ is the initial nutrient concentration.
The nutrient consumption function $f\in C^\infty(\R^+)$ with $\R^+=[0,\infty)$
satisfies the following basic assumption:
\begin{enumerate}[label =\textup{(H\arabic*)}]
\item \label{H1} $f(0)=0$ and $f'(\sigma)\geq 0$ for $\sigma\geq 0$.
\end{enumerate}
The function \(f\) describes the dependence of the nutrient consumption rate on the local nutrient concentration.
Since $f$ cannot be measured directly through medical examinations, yet carries
significant diagnostic and therapeutic implications, its reconstruction is
essential. To describe the motion of the free boundary, let $V_{\mathbf{n}}$ be the
normal velocity of $\Gamma(t)$, that is, the component of the derivative of
$t\mapsto\Gamma(t)$ in the direction of the outward unit normal $\mathbf{n}$.
For each $t>0$ and $x\in\Gamma(t)$, $\kappa_\Gamma\equiv\kappa_{\Gamma(t)}(x)$
denotes the mean curvature of $\Gamma(t)$ at $x$, signed by the convention that
convex hypersurfaces have positive mean curvature.

This paper studies the following forward problem \ref{FP} and inverse problem \ref{IP}, whose mathematical challenges stem from the moving boundary with no a priori information, as well as the nonlinear coupling between the boundary $\Gamma(t)$, the nutrient concentration $\sigma$ and the pressure $\varpi$.

\begin{enumerate}[{\bf (FP)}]
\item\label{FP} Assume that the nonlinearity $f$ is known and satisfies the condition \ref{H1}, and suppose that the initial and boundary conditions in \eqref{eq:DCIS:full} are known. The forward problem is to determine the free boundary $\Gamma(t)$, the nutrient concentration $\sigma(f)$ and the tumor-tissue pressure $\varpi(f)$ in $\Omega(t)$ for $t\in(0,T]$, where $T$ is some fixed time.  
\end{enumerate}

\begin{enumerate}[{\bf (IP)}]
\item\label{IP} Assume that the boundary conditions are known. The inverse problem is to determine the unknown  nutrient consumption rate {$f=f^\dag$} from the needle aspiration biopsy data, i.e., the nutrient concentration
\begin{equation}\label{observation}
g(t):=\sigma(f^\dag;x_0,t) 
\end{equation}
for some $x_0\in \bigcap_{t\in [0,T]} \overline{\Omega(t)}$. 
\end{enumerate}

The analysis of \ref{FP} and \ref{IP} presents several difficulties arising from the interaction between the unknown nonlinearity, the evolving geometry, and the limited observation. First, the nutrient consumption function \(f\) enters a parabolic equation posed on the solution-dependent domain \(\Omega(t)\), while the motion of the free boundary \(\Gamma(t)\) is nonlocally coupled to the nutrient concentration through the elliptic pressure equation and the curvature-dependent boundary condition. Defining the reconstruction operator on a common observation interval therefore requires a forward theory that is uniform over the entire admissible class of consumption functions. In particular, the existence interval and the estimates for the corresponding free boundaries must be independent of the individual choice of \(f\).
Second, the available data consist of a temporal trace of the nutrient concentration at a fixed spatial point, whereas the reconstruction formula involves the pointwise quantity \(\Delta \sigma(f;x_0,\cdot)\). Establishing that this quantity is well defined and depends continuously on \(f\) requires local spatial regularity beyond that provided directly by the basic well-posedness theory. In addition, the nonlinear coupling induced by the evolving free boundary makes the derivation and computation of an adjoint state particularly involved.
To analyze the forward problem, we employ the Hanzawa transformation to map the moving domain onto the fixed reference domain \(\Omega_0\), following the framework developed for free-boundary problems in \cite{cui2007well,cui2009lie}. The elliptic pressure variable is then represented in terms of the transformed nutrient concentration and the free-boundary parametrization, reducing the original system to a quasilinear parabolic evolution equation on a fixed domain. Maximal \(L^p\)-regularity yields a common existence time \(T^*>0\) for all consumption functions in the prescribed admissible class, together with uniform a priori bounds for the associated free boundaries. Both \(T^*\) and these bounds depend only on the parameters defining the admissible class and on the prescribed model data, rather than on a particular choice of \(f\). A localized maximal-regularity argument further provides the spatial regularity needed to evaluate \(\Delta \sigma(f;x_0,\cdot)\) and establishes continuity of the forward observation map with respect to \(f\).

%Fixed-point and projection formulations of inverse coefficient problems have a long history. The contraction-based approach of Pilant and Rundell \cite{PilRun1986, PilRun1988} reformulates the inverse problem as a fixed-point problem for a map acting on the overposed boundary data, and establishes convergence of the iteration in a H\"older norm \(C^\alpha\) with \(\alpha<1\), on the space of Lipschitz nonlinearities. The argument requires  sufficiently small time horizon, and contraction constants that depend on the size and smoothness of the unknown nonlinearity. More recently, Kaltenbacher and Rundell \cite{KalRun2019} developed a projection-based iteration for reaction--diffusion systems, which naturally leads to a fixed-point iteration for the unknown terms; its applicability is tied to the compatibility between the observation manifold and the dependence direction of the unknown, and a convergence proof of Newton-type methods can only be carried out under structural conditions on the forward operator such as the tangential cone condition, which are in general not verified in these settings. To the best of our knowledge, none of these methods apply to inverse coefficient problems posed on evolving free-boundary domain, where the contraction or projection arguments above cannot be verified over the admissible class. 

Inverse coefficient identification for diffusion-type equations on prescribed domains has received considerable attention, including rigorous reconstruction and discretization error analyses for diffusion coefficients in elliptic and parabolic equations, simultaneous recovery
of diffusion and potential coefficients from two observations, and diffusion-coefficient recovery from terminal measurements \cite{CenZhou2024,JinLuQuanZhou2025, JinZhou2021,ZhangZhangZhou2022}. A related line of research employs fixed-point or projection formulations. Pilant and Rundell \cite{PilRun1986,PilRun1988} established convergence of a fixed-point iteration in a H\"older norm under, among other conditions,
a sufficiently small time horizon and a strict contraction estimate depending on the size and regularity of the unknown nonlinearity. Subsequent fixed-point methods have been analyzed for potential
identification from terminal observations \cite{ZhangZhangZhou2022} and for backward semilinear subdiffusion problems \cite{WuYangZhou2026}, while Kaltenbacher and Rundell \cite{KalRun2019} developed a projection-based iteration for reaction-diffusion systems. Although these works provide rigorous theory for inverse problems on fixed domains, the convergence of their iterative schemes generally relies on contractivity, observation geometry, or structural assumptions on the forward operator, and they
do not address the additional nonlinear coupling between an unknown coefficient and an evolving free boundary.
Such assumptions are difficult to verify for the inverse problem of DCIS. Indeed, the unknown consumption function enters the nutrient equation directly and simultaneously affects the evolution of the free boundary through its nonlinear coupling with the pressure equation. Consequently, a sufficiently strong contraction estimate cannot in
general be guaranteed uniformly over the admissible class. Moreover, the forward problem is posed on a solution-dependent domain, so that the construction and repeated evaluation of an adjoint state would be
particularly involved. To the best of our knowledge, existing fixed-point and projection methods do not cover coefficient identification in this coupled free-boundary setting.

This paper addresses this gap by developing a homotopy-continuation method for the fixed-point formulation of the inverse problem introduced in what follows.
For the inverse problem, we use the nutrient equation along the observation curve to construct an inversion operator and combine it with a projection onto the admissible set. This yields an inversion-projection formulation of the inverse problem as a fixed-point problem, in which the unknown consumption function is recovered as a fixed point of the composite iteration operator. The admissible set is subsequently approximated by finite-dimensional spaces, producing a family of discrete iteration operators.
The method has two regimes. At homotopy levels where the map is contractive, the fixed point is computed by Picard iteration, which converges linearly (Lemma~\ref{lem:28}). For the remaining levels, where contraction may fail, we locate the fixed point by a cubical Sperner search over a box in the discrete admissible set. Each evaluation of the iteration operator consists of a forward solver of the free-boundary problem followed by a finite-dimensional projection; hence the method requires only point evaluations of the iteration operator, and neither differentiation of an objective functional nor computation of an adjoint state are required.
%no differentiation of an objective functional and no adjoint-state computation are needed. 
Its convergence is established whenever the discrete iteration operator is continuous, which is the setting of Theorem~\ref{thm:main}. Also note that the fixed-point formulation used here entails a uniqueness statement for the inverse problem (Lemma~\ref{lem:fixedpoit}).
On the analytical side, we prove that the discrete iteration operator admits a fixed point at every discretization level, and that every sequence of discrete fixed points contains a subsequence converging strongly to a fixed point of the continuous iteration operator (Theorem~\ref{thm:main}). Under the stated consistency condition, the limiting fixed point reproduces the measured nutrient trace and hence solves the inverse problem.

The main contributions of this paper are twofold. On the analytical side, we establish uniform local well-posedness of the DCIS free-boundary model over an admissible class of nondecreasing Lipschitz consumption functions, in the sense that all nonlinearities in the class admit a common existence interval and uniform a priori bounds on the free boundary, with constants independent of the particular choice of \(f\) (Theorem~\ref{thm:exist}). We further prove that the discrete iteration operator admits a fixed point at every discretization level, and that every sequence of discrete fixed points admits a subsequence converging strongly to a fixed point of the continuous iteration operator, which solves the inverse problem under a consistency condition (Theorem~\ref{thm:main}); in contrast to the contraction-based approach of \cite{KalRun2019,PilRun1988}, which requires a vanishing initial value and H\"older-type contractivity of the iteration map, no contractivity assumption is made here. On the methodological side, we propose a novel two-regime homotopy-continuation method that combines a Picard iteration with a cubical Sperner search. Meanwhile, 
the method requires only point evaluation of the iteration operator, so that neither differentiation of an objective functional nor computation of an adjoint state is involved.
The effectiveness of the method is demonstrated by numerical experiments on radially symmetric and non-symmetric DCIS models in Section~\ref{sec:num}.

The remainder of the paper is organized as follows. Section \ref{sec:main} introduces the functional setting and the admissible classes, states the main results for the forward and inverse problems, and presents the Sperner-based homotopy-continuation method. Section \ref{sect:reduction} applies the Hanzawa transformation to reduce the moving-boundary system to a quasilinear evolution equation on the fixed reference domain. Section \ref{sec:pf_main}  establishes uniform forward solvability, proves the required regularity and continuity properties of the forward solution map, and derives the fixed-point characterization of the inverse problem together with the existence and convergence of the discrete fixed points. Finally, some numerical experiments are shown in Section \ref{sec:num}.

\section{Main results}\label{sec:main}

In this section, we state the main results of this work and present the proposed algorithm: the uniform existence result for \ref{FP}, the existence and approximation results for \ref{IP}, and a homotopy-continuation method for computing the fixed points of \ref{IP}.

\subsection{Preliminary}
Let  $\Om\subset \R^n$ be a bounded domain. For $1\leq q\leq \infty$ and $k\in \mathbb N$, let 
$W^{k,q}(\Om)$  be the standard Sobolev space. For $k\in \mathbb N$ and $\alpha\in (0,1)$, $C^{k+\alpha}(\overline{\Om})$ and $C^k(\overline{\Om})$ denote the standard H\"{o}lder space and the space of $k$-times continuously differentiable functions defined on $\overline{\Om}$, respectively.  For $\alpha\in (0,1)$ and 
$k\in \mathbb N$, the little H\"{o}lder space $h^{k+\alpha}(\Omega)$ is the closure of $C^\infty(\overline{\Omega})$ 
under the norm $C^{k+\alpha}(\overline{\Om})$. For any $s\in \R$ and $q,r\in [1,\infty]$,  we write $H^{s}_q(\Omega)$ and $B^{s}_{qr}(\Om)$ as the usual Bessel-potential space and  Besov space on $\Om$, respectively.  For a smooth hypersurface $\Gamma\subset \R^n$ of co-dimension 1 and $0\leq s\leq k$, the Besov space  $B^s_{q r}(\Gamma)$ can be defined  with the use of a partition of unity and the norm in $B^s_{q r}(\R^{n-1})$, and  $C^k(\Gamma)$ denotes the space of $k$-times continuously differentiable functions defined on $\Gamma$.  Let $Z$ be a Banach space and $I$ an interval in $\R$.  For $\eta\in (0,1)$ and $k\in \mathbb N$, $C^\eta(I;Z)$ denotes the Banach space of all 
$\eta$-H\"{o}lder continuous functions with values in $Z$ defined on $I$, and $C^k(I;Z)$ is defined as the Banach space of all $k$-times continuously differentiable $Z$-valued functions defined on $I$. 
Subsequently,  we identify  $C(I;Z)=C^0(I;Z)$. For  $1\leq p\leq \infty$ and $s\in \R$,  we denote by  $L^p(I;Z)$ the usual Bochner-Lebesgue spaces of $Z$-valued functions, and by 
$H^s_p(I;Z)$ the corresponding Bessel potential spaces. We also identify $H^s(I;Z):=H^s_2(I;Z)$ and $L^p(I):=L^p(I;\R)$. For a general domain $Q\subset \R^{n+1}$, $L^p(Q)$, $C(Q)$ and 
$C(\overline{Q})$ denote the conventional function spaces, respectively. 
For two Banach spaces $Z_1$ and $Z_2$, we denote by  $\L(Z_1,Z_2)$ the space of all continuous linear operators from $Z_1$ into $Z_2$,  and abbreviate $\L(Z_1,Z_1)$ as $\L(Z_1)$. For two Banach spaces $Z_1$ and $Z_2$, we write 
$
Z_1\embed Z_2 
$ (resp. $Z_1 \stackrel{d}{\embed} Z_2$)
if $Z_1$ is continuously injected in $Z_2$ (resp. densely and continuously  injected in $Z_2$).

For $x\in \R^n$ and $r>0$,  we write $B_r(x):=\{y\in \R^n\mid \|y-x\|_{\R^n}<r\}$.  Let $r>0$ and $\Gamma$ be a non-empty connected compact hypersurface  bounding some region $\Om\subset \R^n$.  We say that 
$\Gamma$ satisfies the $r$-ball condition if there exists a unit vector ${ d}\in \R^n$ such that $B_{r}(x-{rd})\subset \Om$ and $B_{r}(x+{rd})\subset \R^n\backslash \overline{\Om}$ for all $x\in \Gamma$.  Given a set $A\subset \R^n$, the set 
$$
V_r(A):=\{x\in \R^n\mid \textup{dist}(x,A)<r\}
$$
is called the $r$-tubular neighbourhood of $A$.  For two domains $\Om_1,\Om_2$ in $\R^n$, we write $\Om_1\Subset \Om_2$ if $\overline{\Om}_1\subset \Om_2$.

\subsection{Standing assumptions and theoretical results}
Throughout this paper, we assume that there are two simply connected,  bounded  smooth domains $\Om_\Sigma$ and ${\Om}_\Gamma$ such that 
${\Om}_\Gamma\Subset \Om_\Sigma$ and the initial tumor region $\Om_0:=\Om_\Sigma\backslash \overline{\Om}_\Gamma.$  
As a consequence, the initial boundary  $\p\Om_0$ consists of two clopen subsets $\Sigma:=\p{\Om}_\Sigma$  and $\Gamma_e:=\p{\Om}_\Gamma$. In addition, we set the following assumptions on the initial free boundary $\Gamma_e$ and the initial value  $\sigma_0$:  

\begin{enumerate}[label =\textup{(A\arabic*)}]

\item  \label{A1} {$\Gamma_e$ is $C^\infty$-smooth.} 

\item \label{A2} $\sigma_0\in H^2_q(\Om_0)$ with some $q\in (n,\infty)$, {$\sigma_0>0$ on $\overline{\Omega}_0$} and satisfies the  compatibility conditions.

\end{enumerate}
 \begin{figurehere}
  \begin{center}
  \begin{tikzpicture}
    \draw [red,ultra thick] (0,0) circle (2.2cm);
  \fill[orange!40!white] (0,0) circle (2.2cm);
   \fill [white] (0,1.5) % Draws a line
      to [out=0,in=90] (1.5,0) % out, put is the angle control
      to [out=-90,in=45] (1.06,-1.06)
      to  [out=-135,in=-45] (0,-0.2)
      to  [out=180,in=45] (-0.2,-1)
       to  [out=180,in=-45] (-1.06,-0.3)
       to  [out=135,in=-90] (-0.75,0.75)
       to   [out=90,in=-180] (0,1.5);
     \draw [orange,very thick] (0,1.5) % Draws a line
      to [out=0,in=90] (1.5,0) % out, put is the angle control
      to [out=-90,in=45] (1.06,-1.06)
      to  [out=-135,in=-45] (0,-0.2)
      to  [out=180,in=45] (-0.2,-1)
       to  [out=180,in=-45] (-1.06,-0.3)
       to  [out=135,in=-90] (-0.75,0.75)
       to   [out=90,in=-180] (0,1.5);
        % \draw (0,0) node {$O$};
          \draw (-0.3,0.7) node {$\Omega_\Gamma$};
          \draw (-0.63,0) node {\color{RedOrange}$\Gamma(t)$};
           \draw (2.45,-0.1) node {\color{red}$\Sigma$};
           % \draw (-1.0,2.5) node {$\rho(\frac{\pi}{2})$};
          \draw (-0.75,-1.75) node {$\Omega_\Sigma$};
         % \filldraw[red] (-1.7,0) circle (2pt);
            \draw [blue,ultra thick,dashed] (0,0) circle (1.5cm);
            %\draw (-1.8,0.2) node {$\xdag$};
             \draw (-0.20,1.75) node {\color{blue}$\Gamma_e$};
            % \draw [black,dashed] (0,0) circle (2.5cm);
         \draw[black, thick,  ->] (0,-1.45) -- (0,-0.8);
         \draw (0.2, -1.25) node {$\hat{\n}$};
        \end{tikzpicture}
 \caption{\label{illu}\emph{Illustration for the tumor boundaries.}}
  \end{center}

\end{figurehere}

\begin{Remark}\label{remark:ball}
   If $\Omega$ is a $C^{1,1}$-domain, then it satisfies the $r$-ball condition for all sufficiently small $r > 0$ \textup{(}see e.g. \cite[Theorem 2.6]{dalphin2018uniform}\textup{)}. Hence, since $\Omega_0$ is smooth by \ref{A1}, there exists some $\delta_\Omega > 0$ such that
\begin{equation}\label{ball condition}
 \Omega_0 \; \text{satisfies the } 4\delta_\Omega\text{-ball condition and }  \Omega_0 \setminus \overline{V_{\delta_\Omega}(\Gamma_e)} \neq \emptyset.
\end{equation}

\end{Remark}

  For the sake of brevity, $\Gamma_e$ is assumed to be smooth, which can be further weakened to a surface of $B^{3-1/q}_{qq}$-class by the perturbation theory, i.e., $\Gamma_e$ is close to a smooth hypersurface (cf. \cite{cui2009lie}).  A typical example of  such smooth  $\Gamma_e$ is the cylindrical surface, which represents the healthy duct. Under the assumption \ref{A1},  we will characterize the moving surface $\Gamma(t)$ in terms of {\it normal parameterization} over $\Gamma_e$.  More precisely, let $\hat{\n}$ denote the unit outward normal vector of $\Omega_0$ on $\Gamma_e$, see the black arrow in Figure \ref{illu}.
  If there is a function $\rho\in B^{s}_{qq}(\Gamma_e)$ with some $s>0$ such that  a hypersurface $\Gamma_\rho\subset \R^n$ has the form  
\beq\label{surface-class}
\Gamma_\rho=\{\xi+\rho(\xi)\hat{\n}(\xi)\mid \xi\in \Gamma_e\},
\eeq
then $\Gamma_\rho$ is said to be a $B^{s}_{qq}$-normal parameterization over $\Gamma_e$.

%In this section, we state the main results of this work, namely, the uniform existence result for {\bf FP} and the existence and approximation results for {\bf IP}. 
We first present the following theorem which provides a uniform existence time for the solutions of \ref{FP}, independent of the choice of nonlinearity in the admissible class.

\begin{Theorem}[Uniform Existence for {\bf FP}]\label{thm:exist}
Let $q>n$, let \ref{A1}--\ref{A2} hold, and let $\delta_\Omega>0$ be as in \eqref{ball condition}. Then, for every $M_{\mathscr{A}}>0$, there exist $\delta\in(0,\delta_\Omega)$ and $T^*>0$ such that the set
\[
\mathscr{A}:=\{f\in C_b^1(\R^+)\mid  f'(x)\geq 0\quad \textup{and}\quad\|f\|_{W^{1,\infty}(\mathbb R^+)}\leq M_{\mathscr{A}}\}
\]
satisfies the following assertions:
\begin{enumerate}
\item[\textup{(i)}] For every $f\in \mathscr{A}$, the DCIS model \eqref{eq:DCIS:full} has a unique solution $U(f):=(\sigma(f),\varpi(f),\Gamma(f))$ of class $H^2_q\times H^2_q\times B^{4-1/q}_{qq}$ on $(0,T^*]$, and $\Gamma(t)\subset V_{\delta}(\Gamma_e)$ for $t\in [0,T^*]$, where $V_{\delta}(\Gamma_e)$ is the $\delta$-tubular neighbourhood of $\Gamma_e$.
\item[\textup{(ii)}] For $p\in (4,\infty)$, there exists a constant $C_\infty>0$, independent of $f\in \mathscr{A}$, such that
\[
\|\rho(f)\|_{\mathsf{H}_p^1((0,T^*),(B_{qq}^{4-1/q},B_{qq}^{1-1/q}))}\leq C_\infty.
\]
\item[\textup{(iii)}] The nutrient concentration $\sigma(f)$ is non-negative in $[0,T^*]$.
\end{enumerate}
\end{Theorem}

\begin{Remark}

~\begin{enumerate}

\item[\textup{(i)}]  The a priori assumption $f(0)=0$ is not imposed on the admissible set, since our analysis below does not require it; for the sake of exposition, we dispense with it. Our main results remain valid under this additional restriction.

\item[\textup{(ii)}]  In the proof of Theorem \ref{thm:exist}, we shall first extend $f$ trivially to $\mathbb R$.
Then, we will prove that the resulting solution $\sigma(f)$ is always non-negative, and thus only 
the value of $f$ on $\mathbb R^+$ matters.

\end{enumerate}

\end{Remark}

For notational simplicity, we impose $T:=T^*$. We impose the following regularity and monotonicity assumption on the measurement $g$:
\begin{equation*}\label{eq:g}
g\in C^1([0,T])\quad\text{and}\quad g'(t)\geq g_0>0\;\;\text{for all}\;\;t\in[0,T].
\end{equation*}
To recover the unknown nonlinearity $f\in\mathscr{A}$ from $g$, we set
$$
I:=[g(0),g(T)]\quad\text{and}\quad I^{\mathrm{ext}}:=\R^+\setminus I.
$$
Since the measurement determines $f$ only on the interval $I$, we restrict the admissible class to
\begin{equation*}
\mathscr{A}_I:=\{f\in C^1(I) \mid f'(x)\geq 0\quad\textup{with}\quad\|f\|_{W^{1,\infty}(\mathbb R^+)}\leq \epsilon_{\mathscr{A}} M_{\mathscr{A}}\}
\quad \text{for some}\quad \epsilon_{\mathscr{A}}\in (0,1].
\end{equation*}
We further assume that there exists an linear  extension operator $\mathscr{E}:\mathscr{A}_I\to C^1_b(\mathbb R^+)$ such that

\begin{enumerate}[label=\textup{(E\arabic*)}, ref=\textup{(E\arabic*)}]

\item \label{E1} $(\mathscr{E} f)|_I=f$ for all $f\in\mathscr{A}_I$. 

\item \label{E2} $\mathscr{E} f$ is non-decreasing whenever $f$ is non-decreasing.
 
\item \label{E3}  $\mathscr{E}\in \mathscr{L}(C^1(I);C^1_b(\mathbb R^+))$ with $\|\mathscr{E}\|_{\mathscr{L}(W^{1,\infty}(I);W^{1,\infty}(\mathbb R^+))}\leq 1/\epsilon_{\mathscr{A}}$, and $\|\mathscr{E}\|_{\mathscr{L}(W^{2,\infty}(I);W^{2,\infty}(\mathbb R^+))}<\infty$.

\end{enumerate}

\begin{Remark}

~\begin{enumerate}

\item[\textup{(a)}] The conditions \ref{E1} -- \ref{E3} ensure that $\mathscr{E} \mathscr{A}_I\subset  \mathscr{A}$. 

\item[\textup{(b)}] Let us first give an explicit construction.  Let $a_I:=g(0)$ and $b_I:=g(T)$. For $f\in \mathscr{A}_I$, we define the extension $\mathscr{E} f$ by $(\mathscr{E} f)(x)=f(x)$ for $x\in I$, $(\mathscr{E} f)(x)=f(b_I)+f'(b_I)(1-e^{-(x-b_I)})$ for $x\geq b_I$, and $(\mathscr{E} f)(x)=f(a_I)+f'(a_I)(e^{x-a_I}-1)$ for $x\leq a_I$.
    
\end{enumerate}

\end{Remark}

For every $f\in \mathscr{A}_I$, we consider the following truncated DCIS model:
\begin{empheq}[left=\empheqlbrace]{alignat=3}
\partial_t\sigma-\Delta\sigma &= -(\mathscr{E} f)(\sigma), \quad &
-\Delta\varpi &= \mu(\sigma-\sigma^\star)
&&\quad\text{in }\Omega(t),\notag\\[1mm]
V_{\mathbf{n}} &= -\partial_{\mathbf{n}}\varpi, \quad &
\varpi &= \gamma\kappa_\Gamma,\quad \sigma = \sigma_b
&&\quad\text{on }\Gamma(t),\label{eq:DCIS:full:trucated} \\[1mm]
\partial_\n\varpi &= 0, \quad &
\partial_\n \sigma &= 0
&&\quad \text{on }\Sigma,\notag\\[1mm]
\Gamma(0) &= \Gamma_e, \quad &
\sigma(\cdot,0) &= \sigma_0(\cdot)
&&\quad \text{in }\Omega_0.\notag
\end{empheq}

\begin{Remark}
If $\big(f,(\sigma(f),\varpi(f),\Gamma(f))\big)$  solves the truncated inverse problem \eqref{observation} and \eqref{eq:DCIS:full:trucated}, then it also solves the inverse problem \eqref{eq:DCIS:full}--\eqref{observation} with $f$ replaced by $\mathscr{E}f$.
\end{Remark}

To make the reconstruction computationally tractable, we next discretize the admissible set. This is achieved by introducing its finite-dimensional approximation of the admissible set $\mathscr{A}$ as follows:

\begin{Definition}\label{def:admissible}
 The approximating spaces $\{\mathscr{S}_{N}\}_{N\in \mathbb N}\subset C^2(I)$ satisfy the following properties. 
\begin{enumerate}

\item[\textup{(i)}] $\textup{dim}(\mathscr{S}_{N})=N$ and $\mathscr{S}_{N+1}\subset \mathscr{S}_{N}$  for every $N\in \mathbb N$.

\item[\textup{(ii)}] For every $k\in \{1,2\}$, we have the following inverse property: 
\begin{equation*}
 \|f\|_{C^k(I)}\leq C_{k,\mathscr{S}} \|f\|_{L^2(I)}   \quad \forall f\in \mathscr{S}_{N}. 
\end{equation*}

\item[\textup{(iii)}] $\bigcup_{N\in \mathbb N}  \mathscr{S}_{N}$ is dense in $L^2(I)$. 

\end{enumerate}
\end{Definition}

With the discretized admissible sets, we can define the reconstruction operators.  
In the rest of this paper, we define $\mathscr{A}^N_I =\mathscr{A}_I\cap \mathscr{S}_N$, which is a convex set.  For $f\in \mathscr{A}_I$, we denote by $U(f)=(\sigma(f),\varpi(f),\Gamma(f))$ the solution of the DCIS problem \eqref{eq:DCIS:full} and define the \emph{solving operator}
\begin{equation}\label{eq:Sdef}
\mathbb S(f)(t):=-g'(t)+\Delta \sigma(f;x_0,t)\;\;\text{a.e.}\;\,t\in J:=[0,T],
\end{equation}
and the \emph{inversion--projection operator} $\mathbb{P}_N:L^2(J)\to  \mathscr{A}^N_I$ given by
\begin{equation}\label{eq:Pdef}
 \mathbb{P}_N(y):=\arg\min_{f\in \mathscr{A}^N_I}\|f\circ g-y\|_{L^2(J)}^2,
\end{equation}
with $\mathbb P_\infty$ defined by the same formula with $\mathscr{A}^N_I$ replaced by $\mathscr{A}_I$. For the well-posedness of $\mathbb S$ and $\mathbb{P}_N$, we refer to 
Lemma \ref{lem:S} and \ref{lem:P}.
The regularized iteration operator and its formal limit are
\[
\mathbb{T}_N:= \mathbb{P}_N\circ \mathbb S:\mathscr{A}^N_I\to \mathscr{A}^N_I\quad \textup{and}\quad
\mathbb{T}:=\mathbb P_\infty\circ \mathbb S:\mathscr{A}_I \to \mathscr{A}_I.
\]
The connection between the fixed points and the solutions of \ref{IP} is given by the following lemma.
\begin{Lemma}[Equivalence]\label{lem:fixedpoit}
Let $f\in \mathscr{A}_I$ be a fixed point of $\mathbb{T}$ with $\mathbb S(f)\circ  g^{-1}\in \mathscr{A}_I$. Then $\sigma(f; x_0,t)=g(t)$ for all $t\in[0,T]$. In particular, the couple
$\big(f,(\sigma(f),\varpi(f),\Gamma(f))\big)$ solves the inverse problem \eqref{eq:DCIS:full}--\eqref{observation}.
\end{Lemma}

The next result establishes the existence of the fixed points of $\mathbb{T}_N$ and of $\mathbb{T}$, together with the strong convergence of discrete fixed points to a fixed point of $\mathbb{T}$.

\begin{Theorem}[Existence and Convergence for the fixed points]\label{thm:main}
~\begin{enumerate}
\item[\textup{(i)}] For every $N\in \mathbb N$, the set of the fixed points of  $\mathbb{T}_N$ is nonempty. The set of the fixed points of $\mathbb{T}$ is also non-empty. 

\item[\textup{(ii)}] Let $\{f^*_N\}_{N\in \mathbb N}$ be a sequence such that
$f^*_N$ is a fixed point of $\mathbb{T}_N$ for every $N\in \mathbb N$. Then,  $\{f^*_N\}_{N\in \mathbb N}$ contains a subsequence that converges strongly to a fixed point of $\mathbb{T}$.
\end{enumerate}
\end{Theorem}

\subsection{Homotopy continuation method}

Theorem~\ref{thm:main} guarantees the existence of these fixed points, which raises the natural question of how to compute them. As discussed in the introduction, contraction- or projection-based iterations are not applicable to the present free-boundary setting, since the required contractivity or projection conditions cannot be verified over the admissible class; moreover, the computation of an adjoint state is particularly involved for the curvature-driven motion of the interface. We therefore propose a homotopy-continuation method combining two regimes---a Picard iteration where the map is contractive, and a cubical Sperner search where it is not necessarily  contractive---which requires only point evaluations of the iteration operator and no adjoint-state computation.

We seek a fixed point of $\mathbb T_N$ by \emph{homotopy continuation}
\cite{Allgower1980,Allgower1990}. Let $f_0\in\mathscr A_I^N$ be a
chosen anchor and define, for $\alpha\in[0,1]$,
\begin{equation}
  \mathbb T_{N,\alpha}(f) = (1-\alpha)\,\mathbb T_N(f) + \alpha\,f_0,
  \label{eq:homotopy}
\end{equation}
so that $\mathbb T_{N,1}\equiv f_0$ has the known fixed point $f_0$,
while $\mathbb T_{N,0}=\mathbb T_N$ is the target. By the constructive
proof of the Brouwer fixed-point theorem \cite{KelloggLiYorke1976}, for
a generic anchor the zero set of \eqref{eq:homotopy} contains a
continuum joining $f_0$ at $\alpha=1$ to a fixed point of $\mathbb T_N$
at $\alpha=0$.

We track this continuum level by level along a sequence of levels
$\{\alpha_k\}_{k=0}^{K}\subset[0,1]$ with
$1=\alpha_0>\alpha_1>\cdots>\alpha_K=0$, computing at each level $k$
an approximate fixed point of $\mathbb T_{N,\alpha_k}$ warm-started
from the previous solution $f_{N,\alpha_{k-1}}$. The following lemma,
whose proof is deferred to Section~\ref{subsect:homo}, quantifies how
the fixed points at adjacent levels are related.
\begin{Lemma}\label{lem:approxcont}
For every $\alpha\in[0,1]$, the set of the fixed points of $\mathbb T_{N,\alpha}$ is non-empty; moreover, if $f_{N,\alpha}$ is any fixed point of $\mathbb T_{N,\alpha}$, then for every $\alpha'\in[0,1]$,
\begin{equation}
  \bigl\| f_{N,\alpha} - \mathbb T_{N,\alpha'}(f_{N,\alpha}) \bigr\|_{L^2(I)}
  \;\le\; M_{\mathscr{A},L^2}\,|\alpha-\alpha'|,
  \label{eq:approxcont}
\end{equation}
where $M_{\mathscr{A},L^2}:=\mathrm{diam}_{L^2(I)}(\mathscr{A}_I)$
denotes the diameter of the admissible set $\mathscr{A}_I$ in the
$L^2(I)$-norm.
\end{Lemma}
 For large $\alpha$, the homotopy map $\mathbb T_{N,\alpha}$ automatically
satisfies a contraction property, which makes it particularly well suited for
Picard-type iteration: the following lemma shows that it admits a unique
fixed point $f_{N,\alpha}$, and the Picard iteration converges to this fixed
point linearly. Its proof is stated in Section~\ref{subsect:homo}.
\begin{Lemma}\label{lem:28}
If $\alpha$ is sufficiently large, then $\mathbb T_{N,\alpha}$ admits
a unique fixed point and the Picard iteration
$$
f^k_{N,\alpha}:=\mathbb T_{N,\alpha}(f^{k-1}_{N,\alpha})\quad\forall k\in \mathbb N
$$
generates a sequence that converges linearly to the fixed point $f_{N,\alpha}$.
\end{Lemma}

For the remaining levels, where uniqueness or contraction may fail,
we compute fixed points by a robust combinatorial search based on the
Sperner lemma. A fixed point of a continuous self-map on a box is
computed by a pivoting search that requires only point evaluations of
the map. Let $Q\subset\mathscr A_I^N$ be a box and $\mathbb G:Q\to Q$
a continuous map. The box $Q$ is covered by a uniform grid of mesh
$h$, and each grid vertex $v\in Q$ (a vector of coefficients in the
basis of $\mathscr A_I^N$) is assigned the label
\begin{equation}\label{eq:label}
  \ell(v) =
  \begin{cases}
    \min\bigl\{ j\in\{0,1,\cdots,N-1\} : (\mathbb G(v))_j \ge v_j
      \ \text{and}\ v_j < b_j \bigr\}, & \text{if the set is nonempty},\\
    N, & \text{otherwise},
  \end{cases}
\end{equation}
where $(\cdot)_j$ denotes the $j$-th coefficient with respect to the
basis of $\mathscr A_I^N$ and $b_j$ the $j$-th upper bound of $Q$; the
condition $v_j<b_j$ excludes the upper face, so the labeling is
\emph{proper}. By the Sperner lemma
\cite[Thm.~14.4.1]{Yang1999}, every proper labeling admits at least
one \emph{completely labeled} cell, i.e., a cell $\tau$ whose vertex
labels satisfy $\{\ell(v):v\in\tau\}=\{0,1,\cdots,N\}$. The pivoting
search starts from a cell $\tau_0$ whose vertices carry all labels
except $N$ and, as long as the current cell is not completely
labeled, moves to the neighboring cell across the face opposite a
vertex with duplicated label $j$; it stops at a completely labeled
cell $\tau$. The piecewise-linear interpolation of $\mathbb G$ on
$\tau$ admits a fixed point $f\in\tau$, and $f$ satisfies
$\|\mathbb G(f)-f\|\le\mathrm{diam}(\mathbb G(\tau))$
\cite[Ex.~3.14]{Allgower1980}; computing $f$ is purely algebraic and
requires no further evaluations of the map \cite{Eaves1972}. The
residual $\varepsilon=\|\mathbb G(f)-f\|$ is then measured directly;
if it exceeds the tolerance, the grid is refined ($h\leftarrow h/2$)
and the search is restarted from the same cell $\tau$, so that, as
the grid is refined, the extracted points converge to a fixed point of
$\mathbb G$. This procedure is summarized in Algorithm~\ref{alg:cubical}.

\begin{algorithm}[ht]
\caption{Cubical Sperner fixed-point search (pivoting)}
\label{alg:cubical}
\begin{algorithmic}[1]
\Require continuous map $\mathbb G:Q\to Q$, box $Q\subset\mathscr A_I^N$,
  tolerance $\mathrm{tol}$
\State cover $Q$ by a grid of mesh $h$; choose a starting cell
  $\tau_0$ whose vertices carry all labels except $N$
\Repeat
  \State $\tau\leftarrow\tau_0$
  \While{$\{\ell(v):v\in\tau\}\neq\{0,1,\cdots,N\}$}
    \State let $j$ be the duplicated label of $\tau$
    \State $\tau\leftarrow$ neighbor of $\tau$ across the face
      opposite the vertex with label $j$
  \EndWhile
  \State extract $f$ by piecewise-linear interpolation on $\tau$
  \State measure $\varepsilon=\|\mathbb G(f)-f\|$
  \If{$\varepsilon>\mathrm{tol}$}
    \State refine $h\leftarrow h/2$; $\tau_0\leftarrow\tau$
  \EndIf
\Until{$\varepsilon\le\mathrm{tol}$}
\State \Return $f,\varepsilon$
\end{algorithmic}
\end{algorithm}

For $\alpha_k<\underline{\alpha}$, contraction may be lost, and each
remaining level is treated as follows: we first continue the Picard
iteration, warm-started from the previous solution, until its stopping
criterion is met. If, during this iteration, complex behavior such as
guaranteed convergence or divergence is observed---with the residual
serving as the measure of progress---the cubical Sperner search
(Algorithm~\ref{alg:cubical}) is launched from the data point prior to
the Picard iteration, rather than from its terminal point; otherwise it
is launched from the point produced by the Picard step. The search is
implemented on a box centered at that point;  The complete
procedure is summarized in Algorithm~\ref{alg:hybrid}.

\begin{algorithm}[H]
\caption{Homotopy warm-started fixed-point search}
\label{alg:hybrid}
\begin{algorithmic}[1]
\Require anchor $f_0\in\mathscr A_I^N$, tolerance $\mathrm{tol}>0$,
  cutoff $\underline{\alpha}>0$, sequence of levels
  $\{\alpha_k\}_{k=0}^{K}\subset[0,1]$ with
  $1=\alpha_0>\alpha_1>\cdots>\alpha_K=0$
\Ensure $f_{N,\alpha_K}\in\mathscr A_I^N$ and the measured residual
  $\|\mathbb T_N(f_{N,\alpha_K})-f_{N,\alpha_K}\|$
\State $f_{N,\alpha_0}\leftarrow f_0$
\For{$k=1,2,\cdots,K$}
  \State $f\leftarrow f_{N,\alpha_{k-1}}$
  \State run the Picard iteration for $\mathbb T_{N,\alpha_k}$,
    warm-started from $f$ and monitoring the residual, until its
    stopping criterion is met
  \If{$\alpha_k<\underline{\alpha}$}
    \State $f_{\mathrm{init}}\leftarrow f$ if the residual history is
      regular, and employ the pre-Picard point $f_{N,\alpha_{k-1}}$, otherwise
    \State run Algorithm~\ref{alg:cubical} on a box centered at $f_{\mathrm{init}}$
    \If{the search returns an acceptable approximation}
      \State $f_{N,\alpha_k}\leftarrow$ the point returned by the search
    \Else
      \State $f_{N,\alpha_k}\leftarrow f_{\mathrm{init}}$
    \EndIf
  \Else
    \State $f_{N,\alpha_k}\leftarrow f$
  \EndIf
\EndFor
\State \Return $f_{N,\alpha_K}$, $\|\mathbb T_N(f_{N,\alpha_K})-f_{N,\alpha_K}\|$
\end{algorithmic}
\end{algorithm}

\section{Reduction to a fixed  domain}\label{sect:reduction}

   In this section, we transform the DCIS model \eqref{eq:DCIS:full} into a quasilinear parabolic system in the fixed reference domain $\Om_0$.
   For more details on the reduction, we  refer the reader to \cite{cui2007well,cui2009lie} and references therein.

Let $\hbn$ be the outward unit normal field on $\Gamma_e$ and  $\delta\in (0,\delta_\Omega]$ with $\delta_\Omega$ from \eqref{ball condition}. 
Then, the mapping 
\[
\Phi: \Gamma_e\times (-4\delta,4\delta)\to \mathcal{R}:=\textup{Im}(\Phi)\subset \R^n 
\quad \text{with}\quad\Phi(\xi,s)=\xi+s \hbn(\xi)  
\]
is a $C^\infty$-diffeomorphism from  $\Gamma_e\times (-4\delta,4\delta)$ onto $\mathcal{R}$ (cf. \cite[p. 65]{pruss2016moving}). Let us decompose its inverse $\Phi^{-1}\!: \mathcal{R}\to \Gamma_e\times (-4\delta,4\delta)$ by introducing functions $\Xi\!: \!\mathcal{R}\to \Gamma_e$ and  $\Lambda\!\!: \!\mathcal{R}\to(-4\delta,4\delta)$ such that 
$$
(\Xi(x),\Lambda(x))=\Phi^{-1}(x)  \quad \forall x\in \mathcal{R}.  
$$
  We  denote by $\Pi:B^{1-1/q}_{qq}(\Gamma_e)\to H^{1}_q(\Omega_0)$ the linear operator such that $\tilde\rho :=\Pi(\rho)$ is the unique solution of the following boundary value problem: 
\beq\label{extension}
 \Delta \tilde \rho =0  \quad \textup{in}\quad \Omega_0, \quad \tilde \rho=0 \quad \textup{on} \quad \Sigma, \quad  \tilde \rho=\rho \quad \textup{on} \quad \Gamma_e,
\eeq
where $\Omega_0$ is the initial tumor domain.  Here and in what follows, we assume that $q>n$.
By the standard elliptic regularity theory, we have the following boundedness 
\begin{equation}\label{bounded:Pi}
\Pi\in \L(B^{k-1/q}_{qq}(\Gamma_e), H_q^{k}(\Omega_0)) \quad \forall 1\le k\le 3.
\end{equation}
There exists some constant $C_e>0$ such that  \(\|\Pi(\rho)\|_{C^1(\overline{\Om}_0)}\leq C_e\|\rho\|_{B^{1-1/q}_{qq}(\Gamma_e)}\) for all $ \rho\in B^{1-1/q}_{qq}(\Gamma_e)$.
 Assuming further $\delta<\delta_0:=\min\{\delta_\Omega/5,1/C_e\}$,  we define 
\[
\mathcal{O}^0_\delta:=\left\{\rho\in C^1(\Gamma_e)\, \mid \, \|\rho\|_{C^1(\Gamma_e)}<\delta\right\}
\quad \textup{and}\quad \mathcal{O}_\delta:= \mathcal{O}^0_\delta\cap B_{qq}^{3-1/q}(\Gamma_e).
\]
Given $\rho\in \mathcal{O}_\delta$, we denote by $\Gamma_\rho$ the corresponding normal parameterization over $\Gamma_e$ as in \eqref{surface-class}, and by $\Om_\rho$  the domain enclosed by $\Gamma_\rho$ and $\Sigma$. 
The Hanzawa transformation
$\Theta_\rho: \overline{\Omega}_0\to \overline{\Omega}_\rho$
is given by
\begin{equation*}\label{def:Thetarho}
\Theta_\rho(x)=x+\phi(\Lambda(x))\,\Pi(\rho)(x)\,\mathbf{n}(\Xi(x))
\;\;\text{for}\;\; x\in\overline{\Omega}_0\cap\mathcal{R},
\text{ and}\;\;\Theta_\rho(x)=x\;\;\text{otherwise},
\end{equation*}
where $\phi\in C^\infty_0((-3\delta,3\delta),[0,1])$ is  a cut-off function such that $\phi\equiv 1$ on $[-\delta,\delta]$  and $\|\phi'\|_{\infty}<\frac{2}{3\delta}$. 
 As usual, we denote by $\Theta^\rho_*: C(\overline{\Om}_0)\to C(\overline{\Om}_\rho)$ and $\Theta_\rho^*: C(\overline{\Om}_\rho)\to C(\overline{\Om}_0)$ 
 the push-forward operator and the pull-back operator respectively, i.e., 
 \begin{align*}
\Theta^\rho_*(u):=u\circ \Theta^{-1}_\rho  \quad \forall u\in C(\overline{\Om}_0),\quad   \Theta_\rho^*(u):= u\circ \Theta_\rho\quad \forall\, u\in C(\overline{\Om}_\rho).    
\end{align*}
In particular, it follows that 
\begin{equation}\label{conti:Theta:rho}
 [\rho \mapsto \Theta_\rho]\in C^\infty(\mathcal{O}_\delta, H_q^{k}(\Omega_0,\mathbb R^n))   \quad\forall\, 1\leq k\leq 3.
\end{equation}
 Given $\rho\in \mathcal{O}_\delta$, we  define 
$\psi_\rho:\mathcal{R} \to \R$ by 
$
\psi_\rho(x):= \Lambda(x) -\rho(\Xi(x)) 
$ for $x\in \mathcal{R}$.  Clearly $\psi_\rho\in B_{qq}^{3-1/q}(\mathcal{R})$. Then, we define the transformed Laplacian and transformed normal derivative as follows: 
\beq\label{def:Arho}
\mathcal{A}(\rho) u :=  -\Theta_\rho^*\Delta(\Theta_*^\rho u)\;\, 
\forall u\in H^2_q(\Omega_0), \;\;
 \D(\rho) u := (\Theta_\rho\mid_{\Gamma_\rho})^*(\nabla(\Theta_*^\rho u) \cdot \nabla \psi_\rho )
 \;\, \forall u\in C^1(\overline{\Omega}_0),
 \eeq
 where $\Theta_\rho\mid_{\Gamma_\rho}$ denotes the restriction of $\Theta_\rho$ on $\Gamma_\rho$. Our estimates below rely on the estimates of pointwise multiplication and composition operator acting on  Sobolev spaces of fractional order, which can be found in \cite[{Section 4.3.1, Theorem 1}]{runst2011sobolev} and 
  \cite[{Section 5.3.3, Theorem 1}]{runst2011sobolev}.
 
 \begin{Lemma}\label{Lemma:pointwise:super}
~\begin{enumerate}
     \item[\textup{(i)}]  
 For every $s_2\geq s_1\geq n/q$, the following mappings
 \begin{align*}
 &[(a,v)\mapsto av]: H^{s_2}_q(\Omega)\times     H^{s_1}_q(\Omega)\to  H^{s_1}_q(\Omega), \\
&[(a,v)\mapsto av]: B^{s_2-1/q}_{qq}(\Gamma_e)\times     B^{s_1-1/q}_{qq}(\Gamma_e)\to  B^{s_1-1/q}_{qq}(\Gamma_e)
 \end{align*}
 are bilinear and continuous. 
 
  \item[\textup{(ii)}] Let $G_\Omega\in C^\infty(\overline{\Omega}_0\times \R^l;\R)$ and $G_\Gamma\in C^\infty(\Gamma_e\times \R^l;\R)$
  for $l\in \mathbb N$, and let $s>n/q$. Then, the corresponding composition operators 
  \begin{align*}
 &  [(\rho_1,\cdots,\rho_l)\mapsto  G_\Omega(\cdot,\rho_1,\cdots,\rho_l)]: \prod_{i=1}^l H_{q}^{s}(\Omega_0)
 \to H_{q}^{s}(\Omega_0),  \\
 & [(\rho_1,\cdots,\rho_l)\mapsto  G_{\Gamma}(\cdot,\rho_1,\cdots,\rho_l)]: \prod_{i=1}^l B_{qq}^{s-1/q}(\Gamma_e)
 \to B_{qq}^{s-1/q}(\Gamma_e) 
  \end{align*}
  are  infinitely Fr\'{e}chet differentiable.

  \end{enumerate}
  
 \end{Lemma}
 In view of \cite[Lemma 2.2]{Escher1997classical}, we can infer that there exist
 smooth functions  $a_{jk}\in C^\infty(\overline{\Omega}\times \R^n,\R)$, $a_j\in C^\infty(\overline{\Omega}\times \R^n\times \R^{n\times n},\R)$ and $b_j\in C^\infty(\overline{\Omega}\times \R^n;\R)$ such that 
\begin{align}
&\mathcal{A}(\rho)u:=-a^{jk}(\nabla (\Pi{\rho}))\p_{jk} u - a^j(\nabla (\Pi{\rho}),\nabla^2 (\Pi{\rho})) \label{local:expression:A}
\p_j u, \\
&\mathcal{D}(\rho) u:= b^j(\nabla (\Pi{\rho})) \partial_j u,  \label{local:expression:D}
\end{align}
where we have used the Einstein summation convention.
Then, making use of Lemma \ref{Lemma:pointwise:super},  
 \eqref{conti:Theta:rho}, \eqref{local:expression:A} and \eqref{local:expression:D},  we can follow the argument as in \cite[Lemma 2.2]{Escher1997classical}, and infer for all $1\le k\leq 2$ that 
\begin{align}
   \mathcal{A}\in C^\infty(\O_\delta,\L(H_q^2(\Om_0),L^q(\Om_0)))\quad \textup{and}\quad \;\mathcal{D}\in C^\infty(\O_\delta,\L(H^{k}_q(\Om_0),B^{k-1-1/q}_{qq}(\Gamma_e))) \label{prop:AD}.
\end{align}
The transformed mean curvature $\K:C^2(\Gamma_e)\cap \O_\delta(\Gamma_e)\to C(\Gamma_e)$ is introduced as follows:  
\[
\K(\rho)(x)=\textup{the mean curvature of the hypersurface} \,\,\Gamma_\rho\,\textup{at the point} \,\, \Theta_\rho^{-1}(x),
\]
and it has the following decomposition: 
  \beq\label{K:splitting}
  \K(\rho)=\mathcal{L}(\rho)\rho+ \K_1(\rho), 
  \eeq
  where $\mathcal{L}(\rho)$ is a second-order elliptic linear partial differential operator on $\Gamma_e$ with coefficients being functions of $\rho$ and its first-order derivatives, and $\K_1(\rho)$ is a first-order partial nonlinear differential operator on $\Gamma_e$. There exist smooth functions $p_{jk}\in C^\infty(\Gamma_e,\R\times\mathbb R^{n-1}), p\in C^\infty(\Gamma_e,\R\times \mathbb R^{n-1})$ such that,  with respect to any local chart, 
  \begin{equation}\label{expression:L:K}
    \mathcal{L}(\rho)\rho =P^{jk}(\rho,\p \rho)\p_{jk} \rho  \quad\textup{and}
    \quad \mathcal{K}_1(\rho)= P^0(\rho, \partial \rho).
  \end{equation}
Following the lines of the proof of \cite[Lemma 3.1]{Escher1997classical}---which was formulated in little H\"older spaces---and replacing the latter by Besov spaces, we infer from \eqref{expression:L:K} and Lemma \ref{Lemma:pointwise:super} that
\begin{equation}\label{prop:L}
\mathcal{L}\in C^\infty(\mathcal{O}_\delta, \mathscr{L}(B^{4-1/q}_{qq}(\Gamma_e),B^{2-1/q}_{qq}(\Gamma_e)))\quad \text{and}\quad \K_1\in  C^\infty(\mathcal{O}_\delta,B^{2-1/q}_{qq}(\Gamma_e)).
\end{equation}

\noindent Finally,  we introduce the  mapping 
$\M:\O_\delta\times H^{1}_q(\Omega_0)\to L^q(\Omega_0) $ by 
\begin{equation}\label{def:M:op}
\M(\rho,u)(x):=
  -\phi(\Lambda(x))\la (\Theta^*_\rho \nabla \Theta_*^\rho u)(x), \hbn (\Xi(x))\ra_{\R^n} \quad \forall\,x\in \mathbb R^n,
\end{equation}
where $\la \cdot, \cdot\ra_{\R^n}$ denotes the inner product on $\R^n$. By \eqref{bounded:Pi} and \eqref{prop:AD}, we can observe that the above operators are well-defined and smooth (see e.g. \cite[Lemma 3.4]{cui2009lie}) in the sense that  
\begin{align}
\mathcal{M}\in  C^\infty(\O_\delta\times H^k_q(\Omega_0), H^{k-1}_q(\Omega_0)) \quad \forall 1\le k\le 2. \label{op:E:estimate}
\end{align}
With the notations in \eqref{def:Arho}, \eqref{K:splitting} and \eqref{def:M:op}, the DCIS model \eqref{eq:DCIS:full}
 can be transformed as 
\beq\label{forward:transform}
\left\{
 \begin{array}{r@{}ll}
 \p_t u+\A(\rho)u+\M(\rho,u)\Pi(\mathcal{D}(\rho)w) &=-f(u),  \quad
  \A(\rho)w=\mu (u-\sigma^\star)  &  \textup{in}\;\; \Omega_0\times [0,T], \\[1mm]
   \p_t \rho+\D(\rho)w&=0, \quad
   w=\gamma \K(\rho), \quad u=\sigma_b  & \textup{on} \;\; \Gamma_e\times (0,T],\\[1mm]
  \p_\n u=0, \quad  \partial_\n w&=0  &\textup{on} \;\; \Sigma\times (0,T], 
\end{array}\right. 
\eeq
with the initial conditions  
\beq\label{forward:transform:initial}
  u(0)=\sigma_0 \quad \textup{on}\quad \Omega_0 \quad \text{and}\quad
  \rho(0) =0 \quad \textup{on}\quad \Gamma_e. 
\eeq
That is, if $(u,w,\rho)$ is the solution of \eqref{forward:transform}-\eqref{forward:transform:initial}, and letting
$$
\sigma=\Theta^{\rho}_* u, \quad  \varpi=\Theta^{\rho}_* w, \quad \Gamma(t)=\Gamma_{\rho(t)}, \quad \Om(t)=\Om_{\rho(t)},
$$
we obtain a solution $(\sigma, \varpi ,\Gamma)$ of \eqref{eq:DCIS:full} and vice versa (cf. \cite{cui2007well,cui2009lie} for details). 
We will further reduce the problem \eqref{forward:transform}-\eqref{forward:transform:initial} by solving the system for $w$. To this end,  given $\rho\in\mathcal{O}_\delta$, we consider
\beq\label{def:ST}
\mathcal{S}(\rho): L^q(\Om_0) \to  H^2_q(\Om_0) \quad\textup{and}\quad  \T(\rho):  B^{2-1/q}_{qq}(\Gamma_e) \to  H^2_q(\Om_0), 
\eeq
where $w_1:=\mathcal{S}(\rho)g_1$ and  $w_2:=\T(\rho)g_2$ are respectively the solution of boundary value problems: 
$$
\left\{
 \begin{array}{r@{}ll}
\A(\rho)w_1= g_1 & \quad\textup{in}\quad \Om_0, \\
 w_1=0 & \quad \textup{on}\quad \Gamma_e,\\
 \partial_\n w_1=0  & \quad \textup{on}\quad \Sigma
\end{array}\right. \quad\textup{and}\quad  
\left\{
 \begin{array}{r@{}ll}
\A(\rho)w_2= 0 & \quad\textup{in}\quad \Om_0, \\
 w_2=g_2 & \quad \textup{on}\quad \Gamma_e,\\
 \partial_\n w_2=0  & \quad \textup{on}\quad \Sigma.
\end{array}\right. 
$$
Moreover, the operators $\Sr(\rho)$ and  $\T(\rho)$ given by \eqref{def:ST} are well-defined in the sense that 
\begin{align}\label{prop:ST} 
  \Sr\in C^\infty(\O_\delta,  \L(H^{k-2}_q(\Om_0),H^{k}_q(\Om_0)) ),\quad \T\in C^\infty(\O_\delta,  \L(B^{k-1/q}_{qq}(\Gamma_e), H^{k}_q(\Om_0)) )
  \end{align}
  for $k\in\{2,3\}$ (cf. \cite[Lemma 3.6]{cui2009lie}).   With the notations in \eqref{K:splitting} and \eqref{def:ST}, we have  
$$
w=\gamma \T(\rho)\mathcal{L}(\rho)\rho+\gamma \T(\rho)\K_1(\rho)+\mu\Sr(\rho)(u-\sigma^\star)
$$
and thus the system \eqref{forward:transform} can be reduced as follows: 
\beq\label{forward:transform2}
\left\{
 \begin{array}{r@{}ll}
 \p_t u+\A(\rho)u+\mathcal{C}_0(\rho,u)\rho&=\F_0(\rho,u)-f(u) &\quad \textup{in}\quad \Omega_0\times [0,T], \\[1mm]
   \p_t \rho+\B(\rho)\rho&=\G_0(\rho,u),\quad
  u=\sigma_b  &\quad \textup{on}\quad \Gamma_e\times (0,T], \\[1mm]
  \partial_\n u&=\sigma_\Sigma  &\quad \textup{on} \quad \Sigma\times (0,T],  
\end{array}\right. 
\eeq
where 
\begin{align*}
   \B(\rho)\zeta:&= \gamma \D(\rho) \T(\rho)\mathcal{L}(\rho)\zeta,\quad \mathcal{C}_0(\rho,u)\zeta:=\M(\rho,u)\cdot \Pi(\B(\rho)\zeta), \\[1mm]
   \mathcal{G}_0(\rho,u):&=-\D(\rho)[\gamma\T(\rho)\K_1(\rho)+\mu\Sr(\rho)(u-\sigma^\star)],\quad
   \mathcal{F}_0(\rho,u):= \M(\rho,u) \Pi \mathcal{G}_0(\rho,u). 
\end{align*}
To homogenize the boundary condition and for further simplification, we introduce 
$$
u_b \in C^\infty(\Omega\times[0,T]) \quad \textup{such that}\quad u_b\mid_{\Sigma} =\sigma_b \quad\textup{and}
\quad u_b\mid_{\Gamma_e}= \sigma_b, 
$$
and define 
\begin{align*}
 \mathcal{G}(\rho,u):=\mathcal{G}_0(\rho,u+\sigma_b), \quad\mathcal{C}(\rho,u):=\mathcal{C}_0(\rho, u+\sigma_b),\quad
\F(\rho,u):=\F_0(\rho,u+\sigma_b)
\end{align*}
and introduce the mappings 
\begin{align}
\mathbb A(U):=\begin{pmatrix}
 \A(\rho) & \mathcal{C}(\rho,u)\\
 0 & \B(\rho) 
\end{pmatrix}\quad\text{and}  
\quad \mathbb F(U):=\begin{pmatrix}
  \F(\rho,u)-f(u(t)+\sigma_b) \\
  \G(\rho,u)
\end{pmatrix} \quad\forall\, U:=\begin{pmatrix}
    u \\
    \rho 
\end{pmatrix}\in X_1,
\label{def:F} 
\end{align}
where 
$$
\mathbb{X}_1:=E_1^1\times E_1^2:=\{u\in H^2_q(\Omega_0)\mid  u =0 \quad \textup{on}\quad\,\Gamma_e\quad \partial_\n u=0\quad \textup{on}\quad\Sigma\}\times B^{4-1/q}_{qq}(\Gamma_e). 
$$
From \eqref{prop:AD}, \eqref{prop:L}, \eqref{op:E:estimate} and  \eqref{prop:ST},  along with  \eqref{bounded:Pi}, we conclude that
$$
\mathbb A\in C^\infty( \mathcal{O}_\delta,\L(\mathbb{X}_1,\mathbb X)) \quad\textup{and} \quad \mathbb F\in C^\infty(\O_\delta,\mathbb X),
$$
where 
$$
\mathcal{U}_\delta:=B^{3/2}_{qq}(\Omega) \times \O_\delta , \quad 
\mathbb{X}:=E^1\times E^2:=L^{q}(\Omega_0)\times B^{1-1/q}_{qq}(\Gamma_e).
$$
Replacing  $\mathcal{C}_0,\F_0,\G_0$ by $\mathcal{C},\F,\G$  in \eqref{forward:transform2}, and utilizing \eqref{def:F} and the initial conditions \eqref{forward:transform:initial},  we finally arrive at a quasilinear parabolic evolution equation in $X$ as follows: 
\begin{equation}\label{forward:evolutionX}
 \begin{cases}
     U'(t)+\mathbb{A}(U(t))U(t)={\mathbb F}(U(t)) \quad \forall\,t\in (0,T],\\[2mm]
     U(0)=U_0,
 \end{cases}   
\end{equation}
where $U_0:=(\sigma_0-\sigma_b,0)^\top\in X_1$. By a well-known interpolation, we have that 
\begin{equation*}
    \mathbb X_{s,p}:= (\mathbb X_0, \mathbb X_1)_{\gamma,p} \;\;\textup{is equivalent to}\;\;  B_{qp}^{2s }(\Omega_0)\times B_{qp}^{4s-1/q}(\Gamma_e).
\end{equation*}

\section{Proof of the main results}\label{sec:pf_main}
In this section, we first present the proof of Theorem \ref{thm:exist} in Subsection \ref{subsec:41}, which is concerned
with the existence of \ref{FP}. Subsection \ref{subsec:42} is devoted to the proof of Lemma \ref{lem:fixedpoit} and Theorem \ref{thm:main},
which encapsulate the existence and convergence of the fixed point for \ref{IP}.

Our proof relies on the framework of  maximal regularity (see e.g.\cite{amann1995linear}).
Let $1\leq p<\infty$ and $T>0$. Throughout this section, let us write $J=(0,T)$.
For two Banach spaces $\mathsf{X}_0,\mathsf{X}_1$ such that $\mathsf{X}_1\embed \mathsf{X}_0$, we define 
\begin{align*}
\mathsf{H}_p^1(J,(\mathsf{X}_1,\mathsf{X}_0)):=
H_p^1(J;\mathsf{X}_0)\cap L^p(J;\mathsf{X}_1),
\quad {}_0\mathsf{H}_p^1(J,(\mathsf{X}_1,\mathsf{X}_0)):=\{u\in \mathsf{H}_p^1(J,(\mathsf{X}_1,\mathsf{X}_0))\mid u(0)=0 \}.
\end{align*}

\begin{Definition}
Assume that $\mathsf{A}\in  \L(\mathsf{X}_1,\mathsf{X}_0)$ and $J=(0,T)$ for some $T>0$. $\mathsf{A}$ is said to have maximal $L^p$-regularity (on $J$ with respect to $(\mathsf{X}_1,\mathsf{X}_0))$)
if, given any $\mathsf{f}\in L^p(J;\mathsf{X}_0)$, the initial value problem 
$$
\p_t u+\mathsf{A}u=\mathsf{f} \quad \textup{in}\quad J\quad\text{with}\quad u(0)=0
$$
admits a unique solution in ${}_0\mathsf{H}_p^1(J,(\mathsf{X}_1,\mathsf{X}_0))$. Equivalently, it means that 
$$
\partial_t+ \mathsf{A}\quad\textup{is an isomorphism between}\quad {}_0\mathsf{H}_p^1(J,(\mathsf{X}_1,\mathsf{X}_0)) \quad \textup{and}\quad L^p(J;\mathsf{X}_0).
$$
We denote by \( \mathscr{MR}(\mathsf{X}_1,\mathsf{X}_0)\)   the collection of all $\mathsf{A}\in \L(\mathsf{X}_1,\mathsf{X}_0)$  with maximal $L^p$-regularity on $J$. This notation is well-defined because maximal regularity  is independent of both the bounded interval $J$ and the exponent $p\in(1,\infty)$; see \cite[Remark 6.1 (d) and (e)]{Amann2004}. Note that \( \mathscr{MR}(\mathsf{X}_1,\mathsf{X}_0)\) is an open set in the topology of  $\L(\mathsf{X}_1,\mathsf{X}_0)$.

\end{Definition}

\begin{Remark}\label{rem:maximalp}
Let $s\in (0,1+1/q)$.  Given a bounded and smooth domain $U$, consider the operator 
$$
-\Delta: \{v\in B^{s+2}_{qq}(U)\mid \partial_{\bf n} v =0 \,\textup{on}\,\partial U\}\subset B^{s}_{qq}(U)\to B^{s}_{qq}(U).
$$
 Then, the elliptic operator $-\Delta$ enjoys the maximal $L^p$-regularity, which is a classical result of elliptic operator with Neumann boundary condition \cite[Page~279]{pruss2016moving}.
 
\end{Remark}

\subsection{Proof of Theorem \ref{thm:exist}}\label{subsec:41}

To prove Theorem \ref{thm:exist}, we shall make use of the maximal regularity of the governing operator and nonlinearity established in the following lemma.

\begin{Lemma}\label{Lemma:op:MR}

There exists a constant $\delta_0>0$ such that 
\begin{align}\label{MR:AB}
&\mathcal{A}\in C^\infty(\mathcal{O}_{\delta_0}, \mathscr{MR}(E_1^1,E^1_0))\quad \textup{and}\quad\mathcal{B}(\rho)\in C^\infty(\mathcal{O}_{\delta_0}, \mathscr{MR}(E_1^2,E^2_0)). 
\end{align}
In addition, we have that 
\begin{align}
\mathcal{G}\in C^{1}(\mathcal{U}_{\delta_0}, B_{qq}^{1-1/q}(\Gamma_e)), \quad
\mathcal{F}\in C^{1}(\mathcal{U}_{\delta_0}, L^q(\Omega_0)), \quad
\mathbb{A}\in  C^\infty(\mathcal{U}_{\delta_0}, \mathscr{MR}(\mathbb X_1,\mathbb X)) \label{MR:A:mat}.
\end{align}

\end{Lemma}

\begin{proof}
For simplicity, we assume that $\Sigma=\emptyset$.
Fix $\rho\in \mathcal{O}_{\delta_0}$. In view of \eqref{local:expression:A} and Lemma \ref{Lemma:pointwise:super}, 
there exist coefficients $\{a^{jk}\}_{1\leq j,k\leq n}\subset H_q^2(\Omega_0)$ and $\{a^j\}_{j=1}^n\subset H_q^1(\Omega_0)$ such that 
\begin{equation}
 \mathcal{A}(\rho)= -a^{ik} \partial_{ik}  - a^j \partial_j. 
\end{equation}
From the definition \eqref{def:Arho}, it follows that the coefficients $a^{ik}$ satisfy a strong ellipticity condition.  Using the regularity of the  coefficients and 
\cite[Theorem 6.3.2]{pruss2016moving}, we can deduce that $\mathcal{A}\in C^\infty(\mathcal{O}_{\delta_0}, \mathscr{MR}(E_1^1,E^1_0))$. By \eqref{expression:L:K} and Lemma \ref{Lemma:pointwise:super},  we have that 
$\{b^j\}_{j=1}^n\subset H_q^2(\Omega_0)$ and $\{P^{jk}\}_{1\le j,k\leq n-1}\subset B_{qq}^{2-1/q}(\Gamma_e)$.  Using the regularity of the coefficients, to prove the maximal regularity that 
$$
\mathcal{B}(\rho)\in C^\infty(\mathcal{O}_{\delta_0}, \mathscr{MR}(E_1^2,E^2_0)),$$ we only need to follow the process as in \cite{Abels2021} and \cite[Appendix]{ChenLiu2026}. 
The smoothness of $\mathcal{G}$ and $\mathcal{F}$ follows directly from \eqref{bounded:Pi}, \eqref{prop:AD} and \eqref{prop:ST}.
By \eqref{op:E:estimate}, for every $(\rho,u)\in \mathcal{U}_\delta$, 
we obtain   that 
 \begin{align*}
 \|\mathcal{C}(\rho,u)v\|_{L^q(\Omega_0)}\leq & \|\mathcal{M}(\rho,u+\sigma_b)\|_{L^q(\Omega_0)}\|\Pi\mathcal{B}(\rho)v\|_{H_q^1(\Omega_0)} 
 \leq C \|v\|_{B_{qq}^{4-1/q}(\Gamma_e)}
 \quad \forall\,v\in B_{qq}^{4-1/q}(\Gamma_e).
 \end{align*}
Combining \eqref{MR:AB} and \cite[Proposition 5.11]{ChenLiu2026} yields \eqref{MR:A:mat}.

\end{proof}

\begin{proof}[Proof of Theorem \ref{thm:exist}]
 Fix $p>4$, and extend $f$ to $\mathbb R$. 
 Then, for every $\gamma \in (0, 3/4)$,  the  set $\mathcal{U}_{\delta_0}$ is an open subset of $ \mathbb{X}_{\gamma,p}$.
By \ref{Lemma:op:MR}, we can deduce that 
\begin{align}
 \|\mathbb{A}(U)-\mathbb{A}(U')\|_{\mathscr{L}(\mathbb X_1,\mathbb X)}
 \leq& C \|U-U'\|_{\mathbb X_{\gamma,p}}
 \quad \forall\, U,U'\in \mathcal{U}_{\delta_0}\cap \mathbb X_{\gamma,p},\label{unif:Lip:A}\\[1mm]
 \|\mathbb F(U)-\mathbb F(U')\|_{\mathbb X}\leq& C \|U-U'\|_{\mathbb X_{\gamma,p}}  \quad \forall\,  U,U'\in \mathcal{U}_{\delta_0}\cap \mathbb X_{\gamma,p},\label{unif:Lip:F}
\end{align}
where $C>0$ is independent of $f\in \mathscr{A}$.
In view of \eqref{unif:Lip:A} and \eqref{unif:Lip:F}, we may invoke the fixed-point argument of
\cite{ChenLiu2026} (a minor modification of
\cite[Theorem 5.1.1]{pruss2016moving}) to obtain the uniform local existence of solutions. 
The non-negativity of $\sigma(f)$ follows from the comparison principle (see \cite{ChenLiu2026} for more details). This completes the proof.   

\end{proof}

\subsection{Proof of Lemma \ref{lem:fixedpoit}, Theorem \ref{thm:main} and  Lemma \ref{lem:28}}\label{subsec:42}

The purpose of this subsection is to prove Lemma \ref{lem:fixedpoit}, Theorem \ref{thm:main} and Lemma \ref{lem:28}. The proof follows the fixed-point strategy
described above, for which we need the following two lemmas:
Lemma \ref{lem:compo} recalls the well-posedness and the continuity of the composition operators $\sigma \mapsto f(\sigma)$  in the framework of Besov spaces;
Lemma \ref{lem:S} establishes the regularity and continuity of the solution
map $\mathbb S$.

\begin{Lemma}\label{lem:compo}
Let $f\in \mathscr{A}$, and  $\sigma\in B_{qq}^s(U)$ for $s\in (0,1)$. Then, the following assertions hold. 

\begin{enumerate}
    
\item[\textup{(i)}] For all non-negative $\sigma\in B_{qq}^s(U)$,  we have that $\|f(\sigma)\|_{B_{qq}^s(U)}\leq \|f'\|_{L^\infty(\mathbb R^+)}\|\sigma\|_{B_{qq}^s(U)}$. 

\item[\textup{(ii)}] For all non-negative $\sigma,\sigma'\in B_{qq}^s(U)$, we have that 
\begin{align*}
\|f(\sigma_1)-f(\sigma_2)\|_{B_{qq}^s(U)} 
\leq& (\|f''\|_{L^\infty(\mathbb R^+)}+|f'(0)|)
\Big\{ (\|\sigma_1\|_{B_{qq}^s(U)}+\|\sigma_2\|_{B_{qq}^s(U)})\|\sigma_1-\sigma_2\|_{L^\infty(U)} \notag\\
&+(\|\sigma_1\|_{L^\infty(U)}+\|\sigma_2\|_{L^\infty(U)}+1)\|\sigma_1-\sigma_2\|_{B_{qq}^s(U)}
\Big\}. 
\end{align*}

\end{enumerate}

\end{Lemma}
\begin{proof}
Assertion (i) is precisely  {\cite[Theorem 6]{BourdaudSickel2011}}. 
To prove assertion (ii), we consider 
\begin{equation*}
H(x,y):=\frac{f(x)-f(y)}{x-y} - f'(0)\quad\forall x,y\in \mathbb R^+\quad\textup{with}
\quad H(0,0)=0.
\end{equation*}
Then, we have that $\|\partial_i H\|_{L^\infty(\mathbb R^+\times \mathbb R^+)}\leq \|f''\|_{L^\infty(\mathbb R^+)}$ for $1\le i\le 2$. With $s\in (0,1)$, this yields 
\begin{align*}
&\|f(\sigma)-f(\sigma')\|_{B_{qq}^s(U)} 
=\|H(\sigma,\sigma')(\sigma-\sigma')+f'(0)(\sigma-\sigma')\|_{B_{qq}^s(U)}\\
\leq& \|H(\sigma,\sigma')\|_{L^\infty(U)} \|\sigma-\sigma'\|_{B_{qq}^s(U)}+\|H(\sigma,\sigma')\|_{B_{qq}^s(U)} \|\sigma-\sigma'\|_{L^\infty(U)}+|f'(0)|\|\sigma-\sigma'\|_{B_{qq}^s(U)}.
\end{align*}
By (i) and  the decomposition $H(\sigma,\sigma')=H(\sigma,0)+(H(\sigma,\sigma')-H(\sigma,0))$,
\begin{equation*}
\|H(\sigma,\sigma')\|_{B_{qq}^s(U)}\leq 
\|f''\|_{L^\infty(\mathbb R^+)}(\|\sigma\|_{B_{qq}^s(U)}+\|\sigma'\|_{B_{qq}^s(U)}).
\end{equation*}
Similarly, one sees that
\begin{equation*}
\|H(\sigma,\sigma')\|_{L^\infty(U)}\leq 
\|f''\|_{L^\infty(\mathbb R^+)}(\|\sigma\|_{L^\infty(U)}+\|\sigma'\|_{L^\infty(U)}).
\end{equation*}
    
\end{proof}

\begin{Lemma}\label{lem:S}
Let $s\in(n/q,1)$. Suppose that $U\subset \R^n$ is a smooth subdomain  such that 
\begin{equation}\label{def:U}
U\Subset \Om_0\backslash \overline{V_{3\delta}(\Gamma_e)}
\quad\textup{with}\quad \partial U\cap \Gamma_e\neq \emptyset\quad\textup{and}
\quad x_0\in \partial \overline{U}\cap \Gamma_e. 
\end{equation}
 Then, the following
assertions hold, with all constants independent of $f\in \mathscr{A}_I$.
\begin{enumerate}
\item[\textup{(S1)}] \emph{(Well--posedness).} Let $p\in (1,\infty)$. For every $f\in \mathscr{A}_I$,
the solution $(\sigma(f),\varpi(f),\Gamma(f))$ of  the  truncated DCIS model \eqref{eq:DCIS:full:trucated}  satisfies the higher regularity
\begin{equation}\label{sigma:regular}
\|\sigma(f)\|_{L^p(J;B^{2+s}_{qq}(U))}
+\|\sigma(f)\|_{H^1_p(J;B^{s}_{qq}(U))} \le C.
\end{equation}
In particular, the operator $\mathbb S:\mathscr{A}_I\to L^2(J)$ defined in \eqref{eq:Sdef} is well defined.

\item[\textup{(S2)}] \emph{(Continuity).} Let $\{f_k\}_{k=1}^\infty$ be a sequence such that  $f_k\to\widehat f$ in $L^\infty(I)$. Then, it holds that 
$\mathbb S(f_k)\to \mathbb S(\widehat f)$ in $L^2(I)$.

\end{enumerate}
\end{Lemma}
\begin{proof}
\noindent (i) Since $\sigma(f)=\Theta_*^\rho (u(f)+u_b) $ and
$\Theta_*^\rho = \textup{id}$ on $U$, 
Theorem \ref{thm:exist} yields the regularity estimate 
\begin{equation}\label{uniformbd:simga}
\|\sigma(f)\|_{\mathsf{H}_p^1(J,(H_q^2(U),L^q(U)))}\le C.
\end{equation}
Let $\chi:\mathbb R^n\to \R$ be a smooth cut-off function such that $\textup{supp}\,\chi\cap \Omega \Subset \Om_0\backslash \overline{V_{3\delta}(\Gamma_e)} $ and $\chi \equiv 1 $ on  $ U$  in a neighborhood of $x_0$, and define $\sigma_\chi := \sigma \chi$.
This function satisfies the equation
\begin{equation*}
\partial_t \sigma_\chi -\Delta \sigma_\chi = -f(\sigma)\chi+[\chi,\Delta]\sigma \quad \textup{in} \quad U.
\end{equation*}
By Lemma \ref{lem:compo} and \eqref{uniformbd:simga}, it follows that 
$$
\|f(\sigma)\|_{L^p(I;B^{s}_{qq}(U))}\leq M \|\sigma\|_{L^p(I;B^{s}_{qq}(U))}.
$$
Since $[\chi,\Delta]$ is a differential operator of order 1, \eqref{uniformbd:simga} implies that
$$
\| [\chi,\Delta]\sigma\|_{L^p(I;B^{s}_{qq}(U))}\leq M \|\sigma\|_{L^p(I;B^{s+1}_{qq}(U))}.
$$
Maximal regularity on {$B^{s}_{qq}(U)$ now implies that the estimate
\eqref{sigma:regular} holds. The trace continuity $B^{s}_{qq}(U)\embed W^{s-1/q,q}(\partial U)\embed C(\partial U)$} completes the proof.

\noindent (ii) Fix $f_1,f_2\in \mathscr{A}$ and set
$\sigma_i:=\sigma(f_i)$. Let $p\in (4,\infty)$. We first prove the estimate
\begin{equation}\label{claim:estimate:w}
\|\sigma(f_1)-\sigma(f_2)\|_{{}_0\mathsf{H}_p^1((0,t);(H_q^2(U),L^q(U)))}\leq
C \|f_1-f_2\|_{L^\infty(I)}
\quad\forall f_1,f_2\in \mathscr{A}.
\end{equation}
The difference $e:=\sigma_1-\sigma_2$ satisfies the equation
\begin{equation}\label{eq:diffeq}
\partial_t e-\Delta e=F_1+F_2+F_3\quad\text{in} \quad Q_T:=J\times U,
\quad {e(0)=0},
\end{equation}
where
$$
F_1:=\chi\cdot (\mathscr{E} f_1- \mathscr{E} f_2)(\sigma_1),
\quad F_2:=\chi\cdot((\mathscr{E}f_2)(\sigma_1)-(\mathscr{E}f_2)(\sigma_2)),
\quad F_3:=[\chi,\Delta]e.
$$
By maximal regularity on $L^q(U)$ (Remark \ref{rem:maximalp}), it follows that 
\begin{equation}\label{w:MR}
\|e\|_{{}_0\mathsf{H}_p^1((0,t);(H_q^2(U),L^q(U)))}\leq C \sum_{i=1}^3 \|F_i\|_{L^p((0,t);L^q(U))} \quad \forall t\in J,
\end{equation}
where $C>0$ is independent of $t\in J$.
By H\"{o}lder's inequality, the Lipschitz property $\|\mathscr{E}f_i\|_{W^{1,\infty}}\leq M_{\mathscr{A}}$, and Young's inequality,
\begin{align*}
\|F_1(t)\|_{L^q(U)}\le& |\Omega|^{1/q}\|f_1-f_2\|_{L^\infty(I)},
\quad \|F_2(t)\|_{L^q(U)}\le M_{\mathscr{A}}\|{e(t)}\|_{L^q(U)}, \\[1mm]
\|F_3(t)\|_{L^q(U)}\le& M_{\mathscr{A}}\|{e(t)}\|_{W^{1,q}(U)}
\leq \varepsilon \|{e(t)}\|_{W^{2,q}(U)}+ \frac{M_{\mathscr{A}}}{\varepsilon}
\|{e(t)}\|_{L^q(U)}.
\end{align*}
Combining these estimates with \eqref{w:MR} and choosing
$\varepsilon>0$ sufficiently small, we obtain
\begin{equation}\label{w:MR2}
\|e\|_{L^p((0,t);H_q^2(U))}+\|e\|_{H_p^1((0,t);L^q(U))}
\leq C \|f_1-f_2\|_{L^\infty(I)}+\frac{M_{\mathscr{A}}}{\varepsilon}
\|e\|_{L^2((0,t);L^q(U))},
\end{equation}
which gives
\begin{equation*}
\|e\|^2_{H_p^1((0,t);L^q(U))}\leq C \|f_1-f_2\|_{L^\infty(I)}^2+C
\int_0^t \|e\|^2_{H_q^1((0,s);L^q(U))} ds.
\end{equation*}
By Gronwall's inequality, 
$$
\|e\|_{H_p^1((0,t);L^q(U))}\leq C e^{C t} \|f_1-f_2\|_{L^\infty(I)}
\quad \forall\,t\in I.
$$
Substituting this into \eqref{w:MR2}, we obtain \eqref{claim:estimate:w}.
With the estimate \eqref{claim:estimate:w} and the fact that  $f_k\to\widehat f$ in $L^\infty(I)$, we can deduce that 
\begin{equation*}
 \Delta \sigma(f_k)-   \Delta \sigma(f)\to 0 \quad \textup{in} \quad L^2(I;L^2(U)). 
\end{equation*}
On the other hand, the estimate in \textup{(S1)} shows that the sequence $\{  \Delta \sigma(f_k)-\Delta \sigma(f)\}$ is a uniformly bounded set in $L^2(I;B^{s}_{qq}(U))$. Interpolation at $\vartheta\in(0,1)$ gives
$ \Delta \sigma(f_k)-   \Delta \sigma(f) \to0$ in $L^2(0,T;B^{\vartheta s}_{qq}(U))$, and
$B^{\vartheta s}_{qq}(U)\hookrightarrow C(\overline U)$ for $\vartheta s>n/q$ 
yields the claim. Combining this with the definition of $\mathbb S$ completes the proof.

\end{proof}

The following result is standard by convexity and compactness of $\mathscr{A}^N_I$ and 
$\mathscr{A}_I$, respectively (see \cite[Proposition 4.8]{bauschke2017convex}).  
\begin{Lemma}\label{lem:P}

For every $y\in L^2(I)$ the minimizers defining $ \mathbb{P}_N(y)\in \mathscr{A}^N_I$
and $\mathbb P_0(y)\in \mathscr{A}$ exist and are unique. In addition, we have that 
\begin{align*}
\| \mathbb{P}_N(\sigma_1)\circ g-\mathbb{P}_N(\sigma_2)\circ g\|_{L^2(J)}
\le&\|\sigma_1-\sigma_2\|_{L^2(J)},
\quad \forall\sigma_1,\sigma_2\in L^2(J), \\
\| \mathbb{P}_\infty(\sigma_1)\circ g-\mathbb{P}_\infty(\sigma_2)\circ g\|_{L^2(J)}
\le&\|\sigma_1-\sigma_2\|_{L^2(J)},
\quad \forall\sigma_1,\sigma_2\in L^2(J).
\end{align*}

\end{Lemma}
With these preparations, we are in a position to prove  Lemma \ref{lem:fixedpoit} and  Theorem \ref{thm:main}.

\begin{proof}[Proof of Lemma \ref{lem:fixedpoit}]
Set $\tilde \sigma:=\sigma(f)$ and $e(t):=\tilde \sigma(x_0,t)-g(t)$. By
Lemma~\ref{lem:S}~\textup{(S1)}, we have that $\Delta\tilde \sigma(x_0,\cdot)\in
L^2(J)$ and $\partial_t\tilde u(x_0,\cdot)\in L^2(J)$, so  the difference $e$ is
absolutely continuous with $e(0)=\sigma_0(x_0)-g(0)=0$. Evaluating the first equation in \eqref{eq:DCIS:full:trucated} at $x_0$ and subtracting the fixed point identity a.e. on $(0,T)$  gives
$$
f(g(t))=\Delta{\tilde \sigma(x_0,t)}-g'(t) \qquad\text{a.e.}\quad\textup{on }\quad(0,T),
$$ 
which yields  
\[
e'(t)=(\mathscr{E} f)\bigl(g(t)\bigr)-(\mathscr{E} f)\bigl(g(t)+e(t)\bigr)
\qquad\text{a.e.}\quad\textup{on }\quad(0,T).
\]
Since $\mathscr{E} f$ is non-decreasing, we can obtain that 
 $(e^2)'=2e e'\le0$ a.e., and $e(0)=0$ forces $e\equiv0$.
\end{proof}

\begin{proof}[Proof of Theorem \ref{thm:main}]

\noindent \textup{(i)}  The set $\mathscr{A}^N_I$ is a convex and compact subset of a finite dimensional subspace in $L^2(I)$. In view of Lemma \ref{lem:S} and Lemma \ref{lem:P}, the operator $\mathbb{T}_N:\mathscr{A}^N_I \to \mathscr{A}^N_I$  is continuous. 
The Brouwer fixed-point theorem yields the existence of the fixed point of 
$\mathbb T_N$. The same argument also applies to $\mathbb T$.

\noindent \textup{(ii)}  The family $\{f^*_N\}$ is uniformly bounded in 
$W^{1,\infty}$. By
Ascoli--Arzel\`a theorem, every subsequence has a further subsequence, still denoted by $\{f^*_N\}$, such that 
to the limit, so $\widehat f\in \mathscr{A}$. The identity $f^*_N=(\mathbb P_N\circ \mathbb S)(f^*_N)$ is equivalent
to the variational inequality 
$$
( \mathbb S(f^*_N)-f^*_N\circ g,(v-f^*_N)\circ g)_{L^2(J)}\le 0
\quad \forall v\in \mathscr{A}^{k}_I \quad\text{and}\quad \forall k\leq N.
$$
 By Lemma~\ref{lem:S}~\textup{(S2)} and the convergence that $f^*_N \to\widehat f$ in $C(I)$, we can deduce that 
$\mathbb S(f^*_N)\to \mathbb S(\widehat f)$ in $L^2(J)$ and find that 
$$
( \mathbb S(\widehat f)-\widehat f\circ g,(v-\widehat f)\circ g)_{L^2(J)}\le 0
\quad \forall v\in \mathscr{A}^{k}_I\quad\text{and}\quad\forall k\in \mathbb N.
$$
For every $v\in \mathscr{A}$, Definition \ref{def:admissible} (iii) provides $v_k\in  \mathscr{A}_{I}^k$ with
$v_k\to v$ in $L^2(J)$. Passing to the limit and changing
variables gives
$$
( \mathbb S(\widehat f)-\widehat f\circ g,(v-\widehat f)\circ g)_{L^2(J)}\le 0 \quad \forall\, v\in L^2(I), 
$$ 
which implies $\widehat f=\mathbb T(\widehat f)$ and completes the proof.

\end{proof}

\subsection{Proof of Lemma \ref{lem:approxcont} and Lemma \ref{lem:28}}\label{subsect:homo}

\begin{proof}[Proof of Lemma \ref{lem:approxcont}]
By the self-property of $\mathbb T_N:\mathscr{A}_I^N\to \mathscr{A}_I^N$ and its continuity of
$\mathbb T_N$ from Theorem \ref{thm:main}, we can deduce that $\mathbb T_N$ admits at least one fixed point.  
Let $f_{N,\alpha}$ be a fixed point of $\mathbb T_{N,\alpha}$, and let $\alpha'\in[0,1]$. A direct computation yields that 
\begin{equation*}
  \bigl\| f_{N,\alpha} - \mathbb T_{N,\alpha'}(f_{N,\alpha}) \bigr\|_{L^2(I)}
  =(1-\alpha)\|\mathbb T_{N}(f_{N,\alpha})-f_0\|_{L^2(I)},
\end{equation*}
which, together with the inclusions that $f_{N,\alpha},f_0\in \mathscr{A}_I^N$, completes the proof. 
    
\end{proof}

\begin{proof}[Proof of Lemma \ref{lem:28}]
\noindent (Step 1) 
 In this step, we will show that 
\begin{equation}\label{lem:28:claim1}
 \|\mathbb S(f_1) -  \mathbb S(f_2)\|_{L^2(J)}\leq C \|f_1 -f_2\|_{L^2(I)}
 \quad \forall f_1,f_2\in \mathscr{A}_I^N,
\end{equation}
where $C>0$ is independent of $ f_1,f_2\in \mathscr{A}_I^N$. 
Let  $f_1,f_2\in \mathscr{A}_I^N$, and set $\sigma_i=\sigma(f_i)$. Choose $U$ to be the set as in \eqref{def:U}.
Then, the difference
$e:=\sigma_1-\sigma_2$ satisfies the equation \eqref{eq:diffeq}. By
the maximal regularity on $B^{s}_{qq}(U)$,  we can deduce that 
\begin{equation*}
\|e\|_{{}_0\mathsf{H}_p^1((0,t);(B^{s+2}_{qq}(U),B^{s}_{qq}(U)))}\leq C \sum_{i=1}^3 \|F_i\|_{L^p((0,t);B^{s}_{qq}(U))} \quad \forall t\in J.
\end{equation*}
By Lemma \ref{lem:compo}, the uniform boundedness \eqref{sigma:regular}, the continuity of $\mathscr{E}$, and the inverse property in Definition \ref{def:admissible}(ii), we obtain 
\begin{align*}
\|F_1(t)\|_{B^{s}_{qq}(U)}\le& C \|\mathscr{E}(f_1- f_2)\|_{W^{1,\infty}(\mathbb R^+)}\|\sigma_1\|_{B^{s}_{qq}(U)}\leq C\|f_1 -f_2\|_{L^2(J)}, \\[1mm]
\|F_2(t)\|_{B^{s}_{qq}(U)}\le& C \|\mathscr{E}(f_2)\|_{W^{2,\infty}(\mathbb R^+)}(\|\sigma_1\|_{B^{s}_{qq}(U)}+\|\sigma_2\|_{B^{s}_{qq}(U)})\|\sigma_1-\sigma_2\|_{B^{s}_{qq}(U)}\leq C\|w\|_{L^p(J;B^{s}_{qq}(U))},
\end{align*}
where the embedding $B^{s}_{qq}(U)\embed L^\infty(U)$ is used. By the interpolation identity 
$$
B^{s+1}_{qq}(U)=(B^{s}_{qq}(U),B^{s+2}_{qq}(U))_{1/2,q}
$$ 
for every $\varepsilon>0$, it follows  that
\begin{equation*}
\|F_3(t)\|_{B^{s}_{qq}(U)} \leq \varepsilon \|w(t)\|_{W^{2+s,q}(U)}+ \frac{C}{\varepsilon}
\|w(t)\|_{B^{s}_{qq}(U)}.
\end{equation*}
Proceeding as in the proof of Lemma \ref{lem:S} (in particular, applying Gronwall's inequality), we obtain \eqref{lem:28:claim1}. This establishes the Lipschitz estimate.

\noindent (Step 2)  The functional
$$
f\mapsto\|f\circ g-y\|_{L^2(J)}^2=\|f-y\circ g^{-1}\|_{L^2(I)}^2
$$
is continuous and convex on $\mathscr{A}^N_I$, and strictly convex since
$\|w\|_{L^2_\omega(J)}=0$ with $w\in C(J_\delta)$ implies $w\equiv0$.
Existence and uniqueness for $\mathbb{P}_N$ follow from the non-emptiness,
convexity and compactness of $\mathscr{A}^N_I$;
the same argument in the Hilbert space $L^2_\omega(J)$, with $K$
closed and convex. We next record the nonexpansiveness of $\mathbb P_N$:
\begin{equation*}
 \|\mathbb P_N(\sigma_1)\circ g- \mathbb P_N(\sigma_2)\circ g\|_{L^2(J)}\leq \|\sigma_1-\sigma_2\|_{L^2(J)}.
\end{equation*}

\noindent (Step 3) Combining the results in the previous step, we can obtain that 
\begin{equation*}
\|\mathbb T_N(f_1)-\mathbb T_N(f_2)\|_{L^2(I)}\leq C \|f_1- f_2\|_{L^2(I)}
\quad \forall f_1,f_2\in \mathscr{A}_I^N,
\end{equation*}
where $C>0$ is independent of $ f_1,f_2\in \mathscr{A}_I^N$. 
For $\alpha$ sufficiently large, the identity $$\|\mathbb T_{N,\alpha}(f_1)-\mathbb T_{N,\alpha}(f_2)\|_{L^2(I)}=(1-\alpha)\|\mathbb T_N(f_1)-\mathbb T_N(f_2)\|_{L^2(I)}$$ shows that $\mathbb T_{N,\alpha}$ is a contraction on $\mathscr{A}^N_I$; the assertion now follows 
 from the Banach fixed-point theorem.
    
\end{proof}

\section{Numerical simulations}\label{sec:num}
In this section, we present some numerical experiments to demonstrate the effectiveness of the proposed adjoint-free sperner method. The radially symmetric case is shown in Subsection \ref{subsec:rad}, while the non-symmetric case of two dimensions is exhibited in Subsection \ref{subsec:non-rad}. The numerical implementation of the forward problem is itself nontrivial; we refer to \cite{KL2024} and \cite{ChenLiu2026} for its implementation in the radially symmetric and the non-symmetric cases, respectively.

\subsection{Radially symmetric case}\label{subsec:rad}
In this subsection, we present several numerical examples arising in a simplified
DCIS model that support our
theoretical analysis. To this end, we consider the following simplified
DCIS model, which is a precisely radially symmetric case of \eqref{eq:DCIS:full}:
\begin{equation}\label{dcis}
\begin{cases}
 \partial_t \sigma =  \partial^2_r \sigma+ r^{-1}  \partial_r \sigma- f(\sigma)  +F \quad\textup{in} \quad (\varphi(t),R),\\[1mm]
  -\varphi  \partial_t \varphi=\displaystyle\int_{\varphi}^R \mu^\star (\sigma-\sigma^\star)rdr \quad\textup{in}\quad (0,T),\\[1mm]
   \sigma(t,\varphi(t))=\sigma_{\textup{b}},\quad \partial_r\sigma(t,R)=\sigma_{\textup{n}},\\[1mm]
  \sigma(0,\cdot)=\sigma_0(\cdot) \quad\textup{in}\quad (\varphi(0),R),
\end{cases}
\end{equation}
where $F\in C^\infty((0,R)\times\mathbb R)$ is an additional force
term and $\sigma_{\textup{b}},\sigma_{\textup{n}}$ are general smooth
functions. These generalizations allow us to construct smooth
solutions easily while keeping the main results of this paper valid.
In practice, we set
\[
T=1,\qquad \varphi(0)=0.5,\qquad \mu^\star=0.3,\qquad \sigma^\star=1.0,\qquad g(t)=0.5+2.5t^2,
\quad \sigma = g(t)+0.8(2t-1)(1-r),
\]
so that $\varphi(T)=3$ and the relevant range is $I=[0.5,3]$. 
In Example \ref{example:nonlinearity}, we consider three choices of the true nonlinearity $f^\dagger$,
each
of which is recovered in the space $\mathscr S_N$ of polynomials of degree at most $N$, and the mesh-size for time and space is $1/160$.
In addition, we apply Algorithm~\ref{alg:hybrid} to recover the
true nonlinearity $f^\dagger$ with the following settings:
\begin{equation*}
f_0\equiv 1,\quad
\alpha_k=2^{-k}\ \ (k=0,1,\cdots,5),\qquad \alpha_6=0,\qquad \underline{\alpha}=0.02.
\label{eq:settings}
\end{equation*}
%In the sequel, we consider the following three cases. 
\begin{Example}\label{example:nonlinearity}
\textup{(i)} $f^\dag(\sigma)=\sigma^3$, $N=4$;
\;\,\textup{(ii)} $f^\dag(\sigma)=\exp(\sigma/2)-1$, $N=6$;\;\,\textup{(iii)} $f^\dag(\sigma)=\sin \left((\pi \sigma)/6\right)$, $N=7$.
\end{Example}
\noindent To measure the error between the true nonlinearity and the recovered one,  we employ the relative $L^2$-error:
\begin{equation}\label{def:e2}
  e_2:= \|f^\dag - f\|_{L^2(I)}/\|f^\dag\|_{L^2(I)}.
\end{equation}
Figure~\ref{fig:dcis_results} shows the reconstruction results for Example \ref{example:nonlinearity}. In each case, the recovered nonlinearity (red dashed curve)
is visually indistinguishable from the true one (black solid curve). The
bottom row shows the relative $L^2$-error $e_2$ with the homotopy level
$k$: after a steady decrease in the Picard regime ($k\le 5$), the
error decays sharply at the final level $k=6$, where the Sperner search
at $\alpha_6=0$ is applied. This sharp decay is explained by the homotopy construction: the fixed points of $\mathbb T_{N,\alpha}$ usually differ for different levels $\alpha$, and $e_2$ is measured against the fixed point at $\alpha=0$; only at the final level does the map reduce to $\mathbb T_N$, whose fixed point approximates the true nonlinearity $f^\dagger$.

\begin{figurehere}
\centering
\begin{minipage}[t]{0.30\textwidth}
  \centering
  \includegraphics[width=\linewidth]{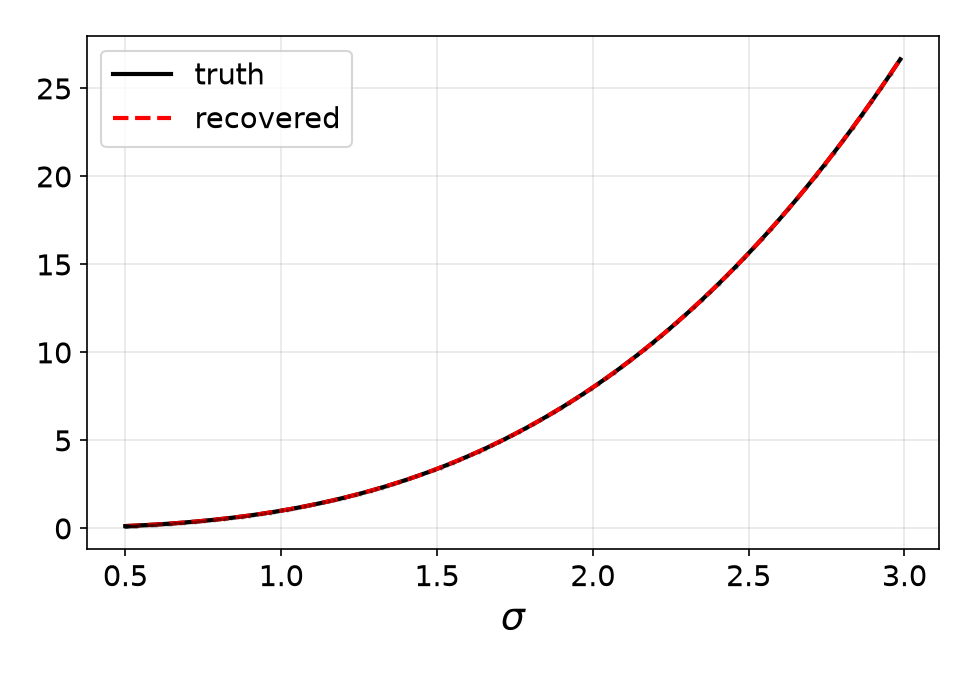}
\end{minipage}\hfill
\begin{minipage}[t]{0.30\textwidth}
  \centering
  \includegraphics[width=\linewidth]{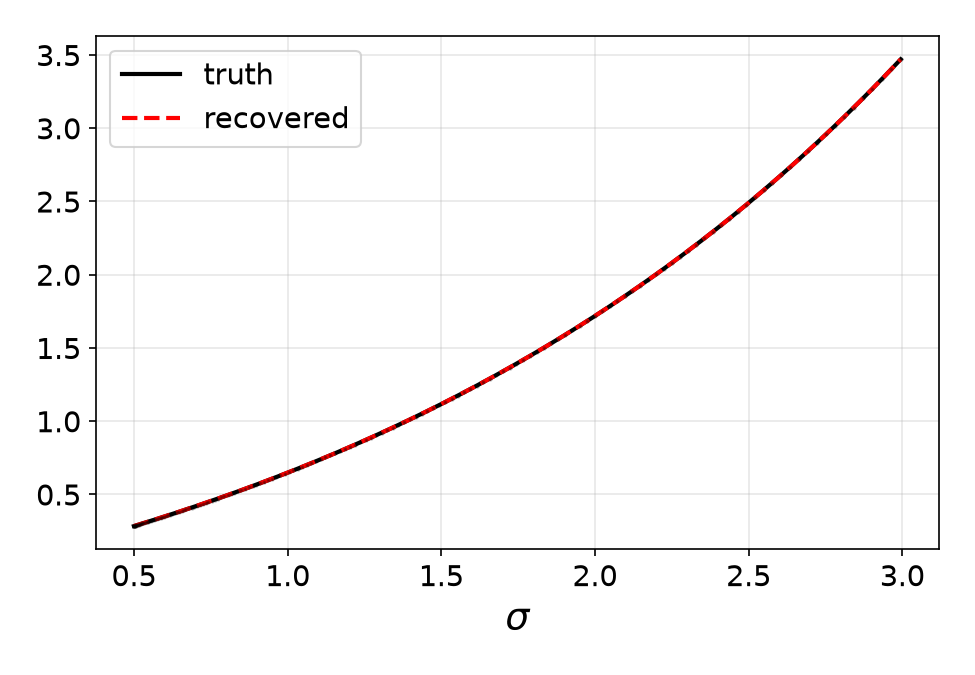}
\end{minipage}\hfill
\begin{minipage}[t]{0.30\textwidth}
  \centering
  \includegraphics[width=\linewidth]{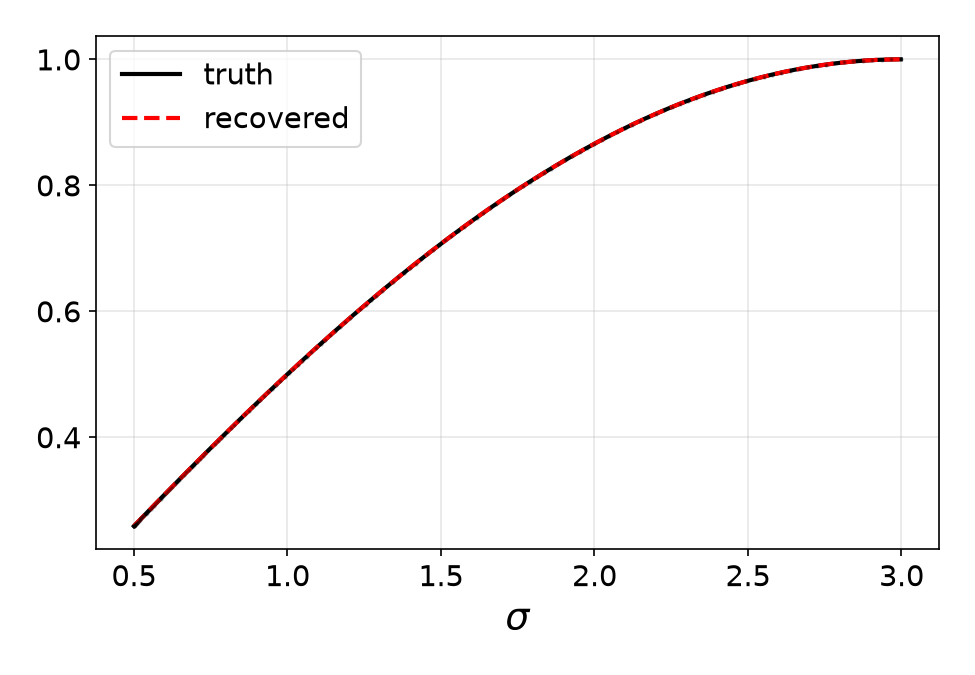}
\end{minipage}
\begin{minipage}[t]{0.33\textwidth}
  \centering
  \includegraphics[width=\linewidth]{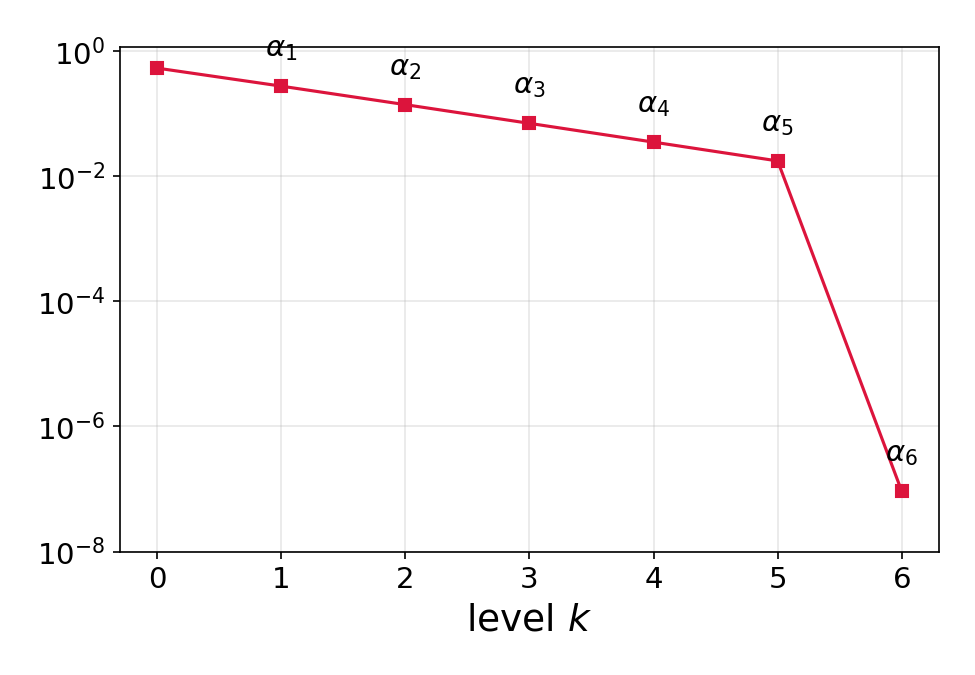}
\end{minipage}\hfill
\begin{minipage}[t]{0.33\textwidth}
  \centering
  \includegraphics[width=\linewidth]{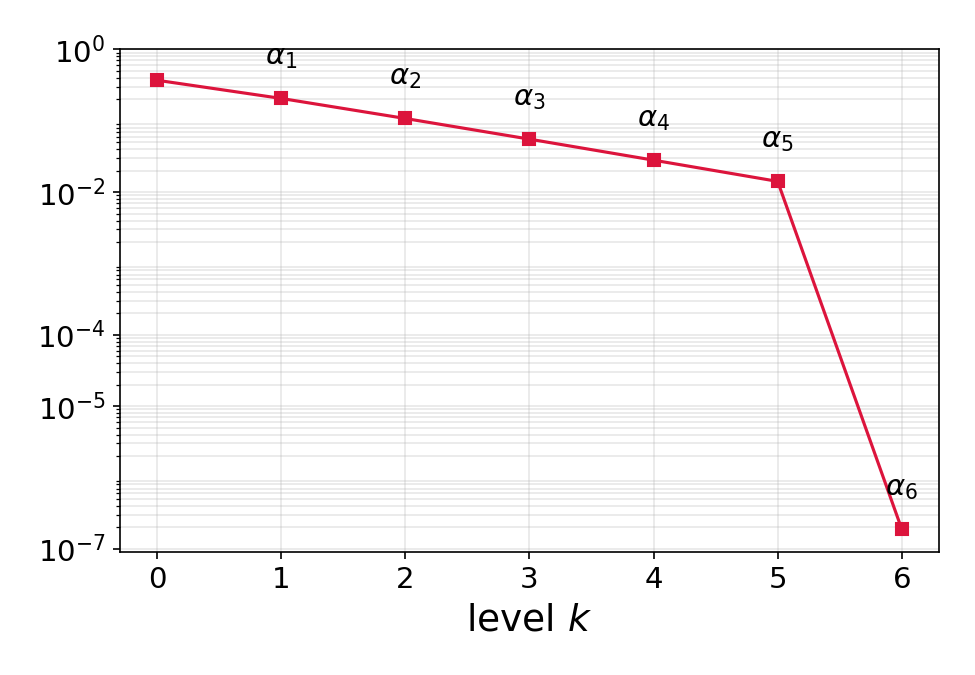}
\end{minipage}\hfill
\begin{minipage}[t]{0.33\textwidth}
  \centering
  \includegraphics[width=\linewidth]{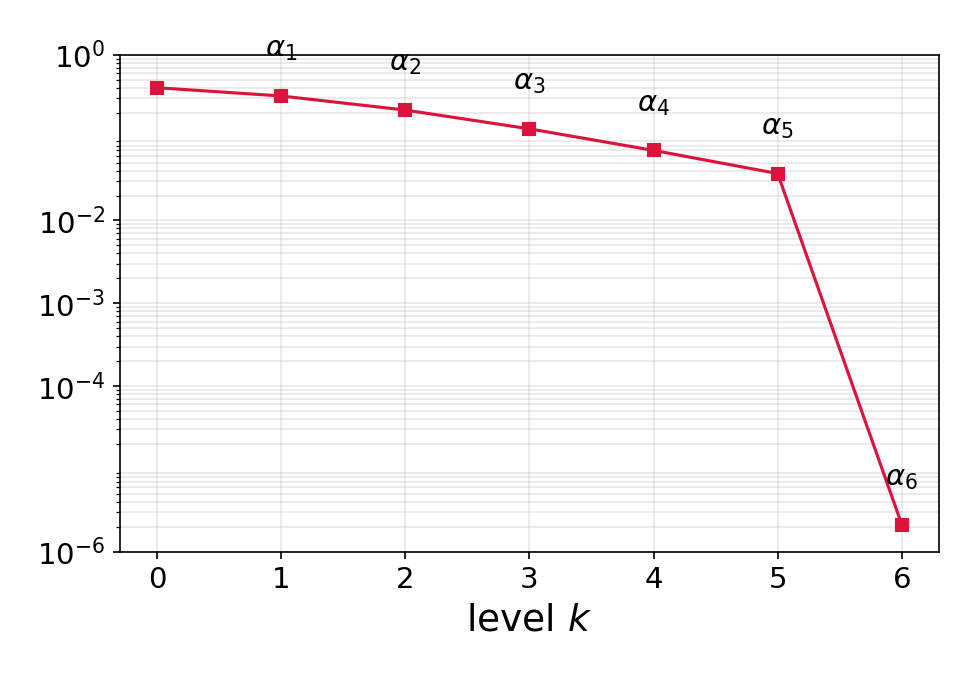}
\end{minipage}

\caption{\label{fig:dcis_results}\emph{Reconstruction results for Example \ref{example:nonlinearity}.
Top row (left to right): recovered nonlinearity (red dashed curve) versus
the true one (black solid curve). Bottom row (left to right): the
corresponding $L^2$-error $e_2$ versus the homotopy level $k$ (log scale);
the reduction factors are marked above the data points.}}
\end{figurehere}

\subsection{Non-symmetric case in two dimensions}\label{subsec:non-rad}

In this subsection, we assess the proposed reconstruction method in a two-dimensional, non-symmetric setting using the method of manufactured
solutions (MMS). The MMS~\cite{banks2019high} provides a controlled benchmark for verifying the numerical implementation and evaluating the reconstruction accuracy, particularly when closed-form solutions of the original nonlinear free-boundary problem are not
available. Accordingly, we augment system~\eqref{eq:DCIS:full} with suitable
manufactured source terms and obtain the following modified system:
%In this subsection, we will demonstrate the reconstruction of the nutrient consumption rate for the non-symmetric case in two dimensions via the method of manufactured solutions (MMS). The MMS~\cite{banks2019high} is a verification technique in which source terms are added to the governing equations so that a chosen analytical function becomes an exact solution. This approach is particularly valuable for nonlinear or complex systems where exact solutions to the unforced equations cannot be constructed. With MMS, the system \eqref{eq:DCIS:full}  becomes in the following form:
\begin{equation}\label{sys:MMS}
\left\{
\begin{aligned}
\p_t \sigma - \Delta \sigma &= -f(\sigma) + S_\sigma,\quad -\Delta \varpi = \mu(\sigma - \sigma^\star) + S_\varpi && \text{in}\quad \Omega(t), \\
V_{\n} &= -\partial_{\n} \varpi + S_\Gamma,\quad \varpi = \gamma \kappa_\Gamma + S_\kappa, \quad \sigma = \sigma_b && \text{on}\quad \Gamma(t), \\
\partial_{\n} \varpi &= 0, \quad \partial_{\n} \sigma = 0 && \text{on}\quad \Sigma, \\
\Gamma(0) &= \Gamma_e, \quad \sigma(\cdot, 0) = \sigma_0(\cdot) && \text{in}\quad\; \Omega_0,
\end{aligned}
\right.
\end{equation}
where $S_\sigma$, $S_\varpi$, $S_\Gamma$ and $S_\kappa$ are the added source terms, and the parameters are selected as $\mu = 1$, $T=1$, $\sigma^\star = 0$ and $\gamma = 1$. The initial condition is set as $\sigma_0=0$, the boundary condition as $\sigma_b=c(t)(\cos \theta+1)\,(2-(2-r)^2)/4\mid_{\Gamma(t)} $ with 
$c(t)=(1-e^{-t})/(1-e^{-1})$. In addition, we choose 
\begin{equation}
\sigma = c(t)(\cos \theta+1)\,(2-(2-r)^2)/4, \quad   x_0 = (0,2), 
\quad I=[0,1]. 
\end{equation}

\noindent We still use the same $\mathscr S_N$ as in the previous case with $N=4$, and the
mesh-sizes for time and space are $1/80$ and $1/40$, respectively.  In addition, we apply Algorithm~\ref{alg:hybrid} to recover the
true nonlinearity $f^\dagger$ with the following settings:
\begin{equation*}
f_0\equiv 0.5,\quad
\alpha_k=2^{-k}\ \ (k=0,1,\cdots,10),\qquad \alpha_{11}=0,\qquad \underline{\alpha}=0.2.
\label{eq:settings2}
\end{equation*}

\begin{Example}\label{example:multi}
\textup{(i)} $f^\dag(\sigma)=\sigma^3$;
\quad \textup{(ii)}  $f^\dag(\sigma)=\exp(\sigma/2)-1$. 
\quad \textup{(iii)}  $f^\dag(\sigma)=1-e^{-2\sigma}$. 
\end{Example}

Figure~\ref{fig:inv_results} displays the reconstruction results for Example \ref{example:multi}: in each case, the recovered nonlinearity (red dashed curve) coincides with the true one (black solid curve). The bottom row reports the relative $L^2$-error $e_2$ against the homotopy level $k$ (log scale). After a steady decrease in the Picard regime ($k\le 10$), the error drops sharply at the final level $k=11$, where the Sperner search at $\alpha_{11}=0$ is applied, in the second and third cases, whereas no such sharp drop occurs in the first case. The mechanism is the same as in the radially symmetric case of Subsection \ref{subsec:rad}: the fixed point of $\mathbb T_{N,\alpha}$ varies with the level $\alpha$, and $e_2$ is measured against the fixed point at $\alpha=0$.

\begin{figurehere}
\centering
\begin{minipage}[t]{0.30\textwidth}
  \centering
  \includegraphics[width=\linewidth]{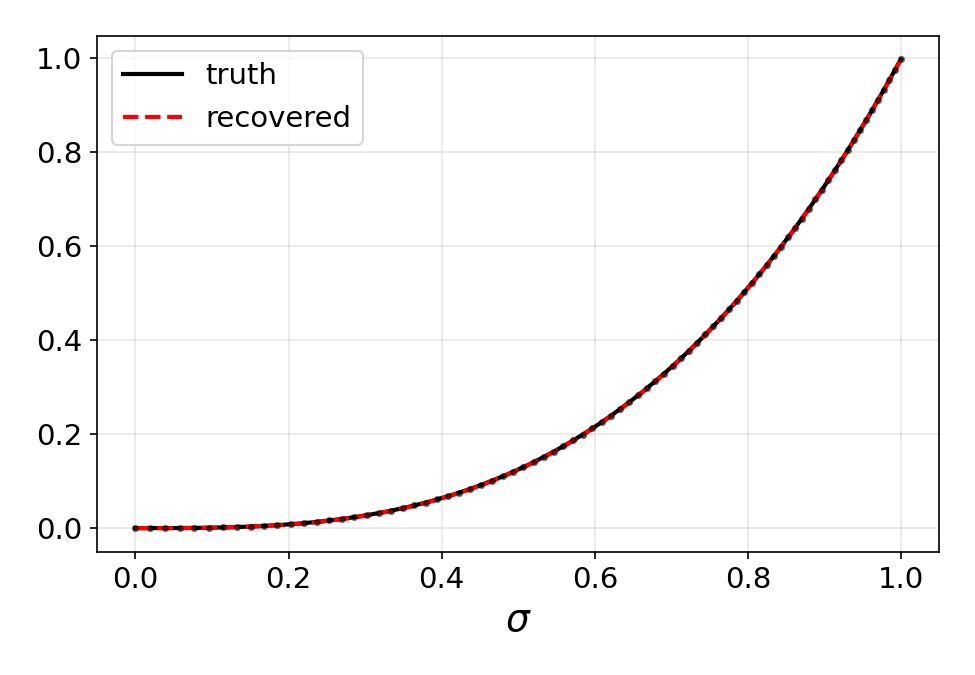}
\end{minipage}\hfill
\begin{minipage}[t]{0.30\textwidth}
  \centering
  \includegraphics[width=\linewidth]{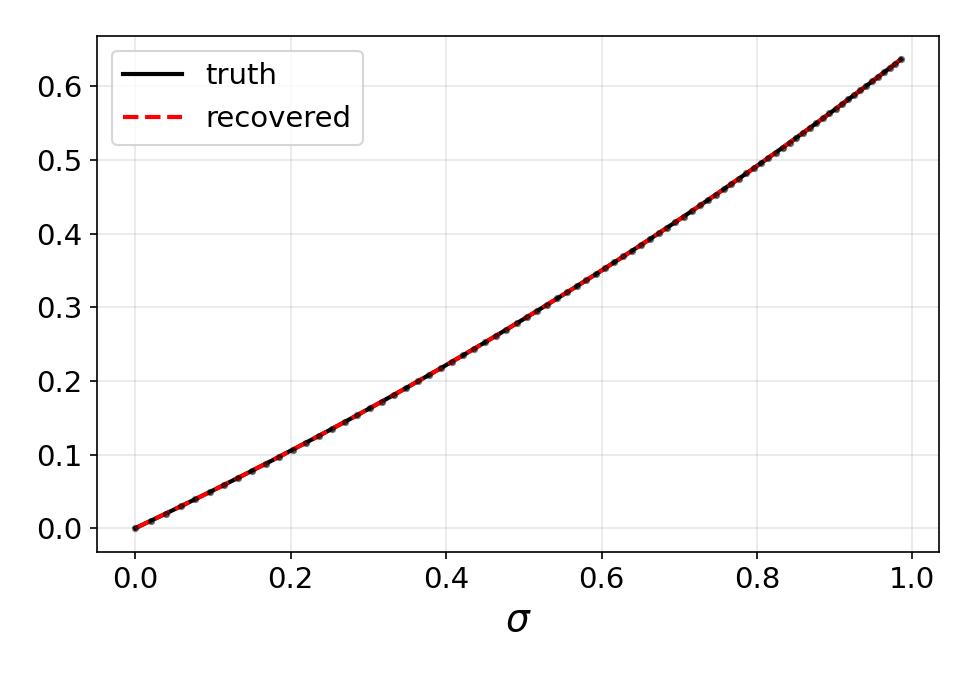}
\end{minipage}\hfill
\begin{minipage}[t]{0.30\textwidth}
  \centering
  \includegraphics[width=\linewidth]{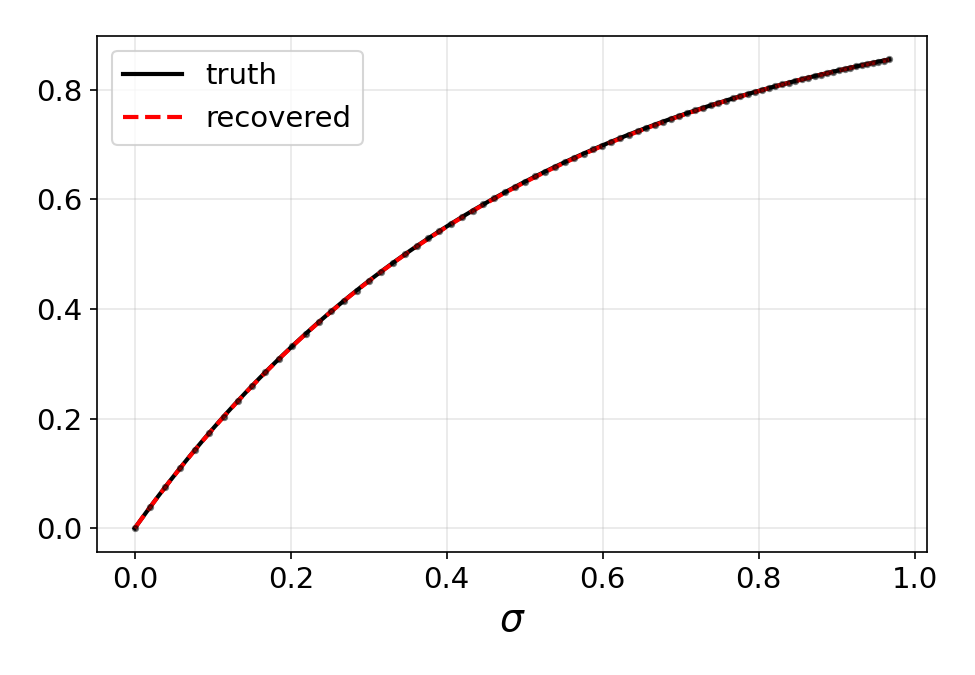}
\end{minipage}
\begin{minipage}[t]{0.33\textwidth}
  \centering
  \includegraphics[width=\linewidth]{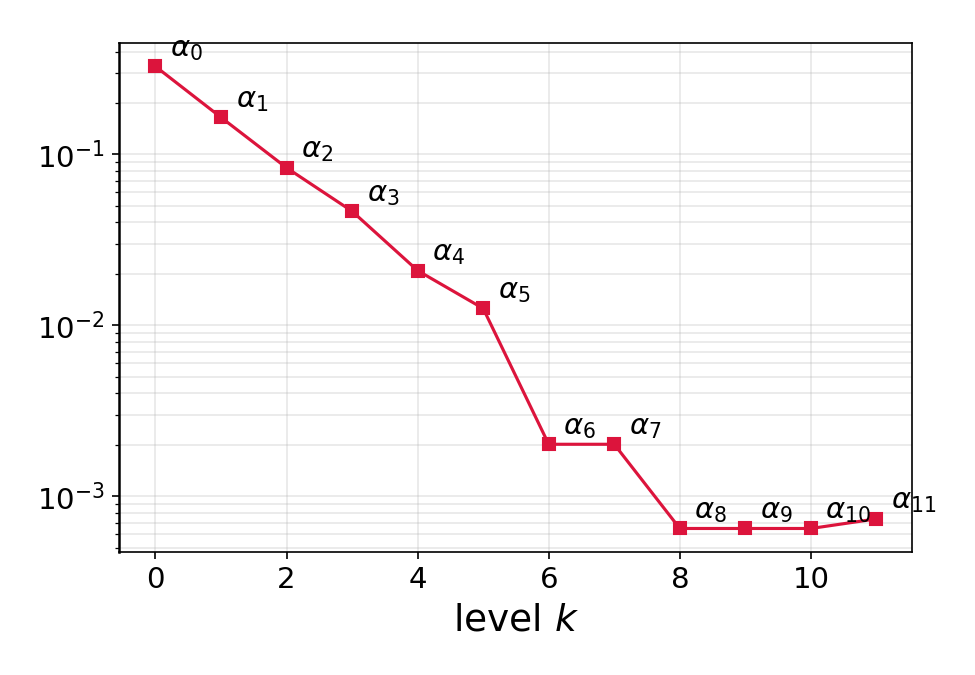}
\end{minipage}\hfill
\begin{minipage}[t]{0.33\textwidth}
  \centering
  \includegraphics[width=\linewidth]{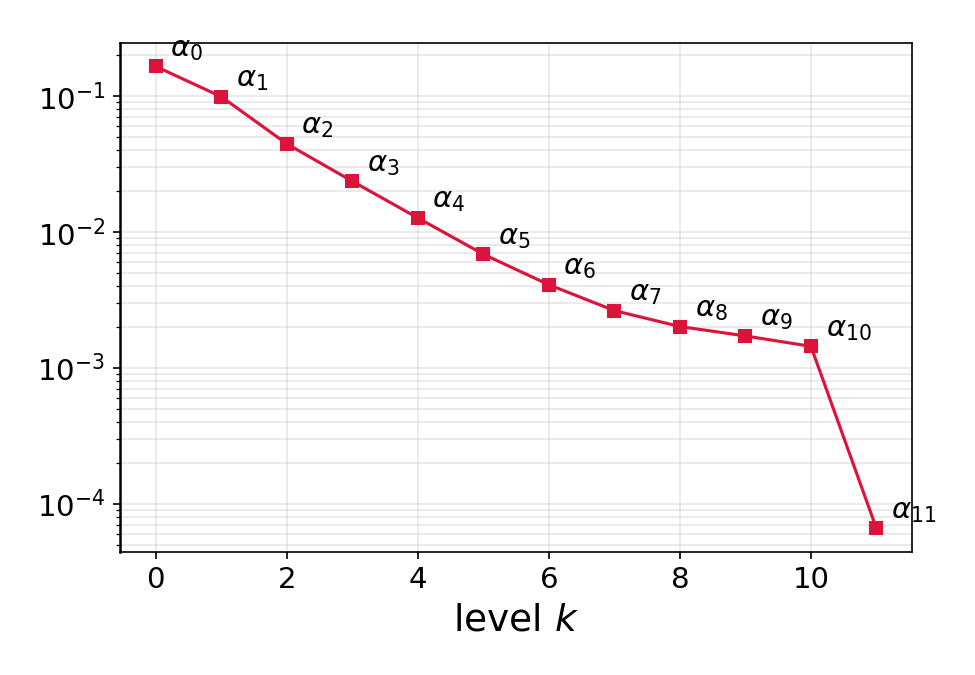}
\end{minipage}\hfill
\begin{minipage}[t]{0.33\textwidth}
  \centering
  \includegraphics[width=\linewidth]{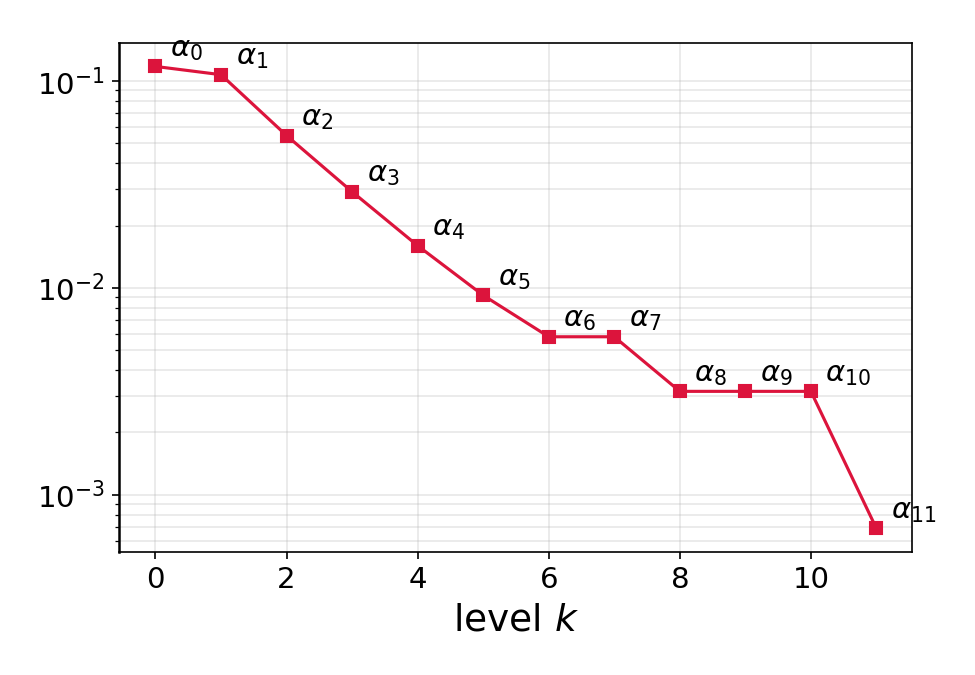}
\end{minipage}

  \caption{\em Reconstruction results for Example \ref{example:multi}.
Top row (left to right): recovered nonlinearity (red dashed curve) versus
the true one (black solid curve). Bottom row (left to right): the
corresponding $L^2$-error $e_2$ versus the homotopy level $k$ (log scale);
the reduction factors are marked above the data points.}  \label{fig:inv_results}

\end{figurehere}

\bibliography{dehansto}{}
\bibliographystyle{acm}

 \end{document}